\documentclass [12pt] {article}
\usepackage{amsthm}
\usepackage{enumerate}
\usepackage{amsmath}
\usepackage{amsfonts}
\usepackage{amssymb}
\usepackage[all]{xy}
\usepackage{color} 
\usepackage{tikz} \usetikzlibrary{arrows}

\usepackage{listings}
\usepackage{setspace}
\usepackage[font=small,labelfont=bf]{caption}
\singlespace

\def\TC{\protect\operatorname{TC}}
\def\pr{\protect\operatorname{pr}}
\def\er{\protect\operatorname{er}}

\def\zcl{\protect\operatorname{zcl}}
\def\cat{\protect\operatorname{cat}}
\def\hdim{\protect\operatorname{hdim}}
\def\cl{\protect\operatorname{cl}}
\newcommand{\overbar}[1]{\mkern 1.5mu\overline{\mkern-1.5mu#1\mkern-1.5mu}\mkern 1.5mu}

\newtheorem{ejem}{Example}
\newtheorem{defi}[ejem]{Definition}
\newtheorem{teo}[ejem]{Theorem}
\newtheorem{prop}[ejem]{Proposition}
\newtheorem{lema}[ejem]{Lemma}
\newtheorem{remark}[ejem]{Remark}
\newtheorem{coro}[ejem]{Corollary}

\numberwithin{ejem}{section}

\begin{document}

\author{Jorge Aguilar-Guzm\'an, Jes\'us Gonz\'alez and Teresa Hoekstra-Mendoza}
\title{Farley-Sabalka's Morse-theory model and the higher topological complexity of ordered configuration spaces on trees}
\maketitle
\begin{abstract}
Using the ordered analogue of Farley-Sabalka's discrete gradient field on the configuration space of a graph, we unravel a levelwise behavior of the generators of the pure braid group on a tree. This allows us to generalize Farber's equivariant description of the homotopy type of the configuration space on a tree on two particles. The results are applied to the calculation of all the higher topological complexities of ordered configuration spaces on trees on any number of particles.
\end{abstract}

\small {\it 2010 Mathematics Subject Classification}: Primary: 20F36, 55R80, 57M15; Secondary: 55M30, 57Q10, 68T40.

\small {\it Keywords and phrases}: Discretized configuration space on a tree, discrete Morse theory, Farley-Sabalka gradient field, topological complexity.

\section{Introduction}
Michael Farber proved in \cite{mp} that the ordered configuration space of two particles on a tree has the homotopy type of a banana graph, i.e., a graph with two vertices and no loops. Farber's argument uses classic algebraic topology methods. The present paper arose from the observation that Farber's result can be recovered and generalized from techniques in discrete Morse theory, as implemented in~\cite{fs} for unordered configuration spaces on graphs. Farley-Sabalka's discrete Morse theory approach is particularly strong: configuration spaces of graphs are known to be aspherical, and Farley-Sabalka's method provides us with an algorithmic way to obtain a presentation of their fundamental groups (see~\cite{MR2949126}).

Farley-Sabalka's model was developed for unordered configuration spaces. The existence of a corresponding model in the ordered configuration case was noticed in Safia Chettih's Ph.D.~thesis \cite[Proposition~2.2.3]{sc} without providing (or making use of) an explicit description. We start by extending Farley-Sabalka model to the realm of ordered configurations. The construction presented here parallels that in \cite{fs}, though we slightly reorganize and streamline a few of the arguments in~\cite{fs}. As in the original (unordered) case, having an explicit description of Farley-Sabalka's model in the ordered situation has the advantage of providing us with easy means to understand, in combinatorial and topological terms, a number of algebro-topological arguments in the literature about ordered configuration spaces on graphs. For example, it will be transparent that the ordered version of Farley-Sabalka's discrete model is equivariant with respect to the coordinate-permutation action of the symmetric group. In particular, Farley-Sabalka's techniques allow us to generalize Farber's \emph{equivariant} description~\cite[Theorem~11.1]{mp} of the homotopy type of the 2-particle ordered configuration space on a tree. Indeed, the equivariance in Farley-Sabalka's model for the ordered situation means that we get a complete description of the monodromy in the usual covering map from ordered to unordered configurations.

We use the resulting tool to address combinatorial properties in the 1-skeleton of the Morse theoretic model of an $n$-particle ordered configuration space on a tree. In particular, this yields a description of the equivariant homotopy type of ordered configuration spaces on trees with at most three particles. For configuration spaces with more particles, the relevance of our work is better understood in terms of~\cite[Theorem 2.5]{fs}, where Farley and Sabalka describe a discrete Morse theory method to produce a presentation for the fundamental group of a given simplicial complex $X$ equipped with a discrete gradient field. The method involves the study of certain complexes $X'_i$, $i=1,2$, which capture essential (gradient-field type) information of the $i$-th skeleton of $X$: $X_1'$ ($X_2'$) provides generator-type (relation-type) information for a (Morse-theory algorithmic) presentation of $\pi_1(X)$. In these terms, we give in Section~\ref{secciondecombinatoria} a complete description of the combinatorio-geometric properties of $X'_1$ when $X$ is Farley-Sabalka's discrete Morse model for the ordered configuration space on a tree on any number of particles.

We also use Farley-Sabalka's ordered discrete model in order to compute the Lusternik-Schnirelmann category (cat) and all the higher topological complexities (TC$_s$, $s\geq2$) of the ordered configuration space on a tree. Such a result was originally obtained by Farber \cite{mp} in the case of cat and TC$_2$ under the additional hypothesis that the number of particles is no smaller than twice the number of essential vertices of the tree. The unrestricted result (for TC$_2$) was then obtained in \cite{LRm} by consideration of graph configuration spaces with sinks, i.e., where collisions on certain vertices are allowed. See also~\cite{MR3773741}, where the $\TC_2$ calculation is carried over (both ordered and unordered) configuration spaces on trees with $n$ particles, for $n$ satisfying certain technical conditions. Our approach (i.e., the direct generalization of Farber's original argument without using configurations with sinks) is both conceptually and computationally simpler, even in the case of higher topological complexity.


\section{Preliminaries}
\subsection{Topological complexity}
For $s\geq2$, the $s$-th topological complexity of a path-connected space $X$, $\TC_s(X)$, is defined as the sectional category of the evaluation map $e_s\colon PX\to X^s$ which sends a (free) path on $X$, $\gamma\in PX$, to the $s$-tuple$$e_s(\gamma)=\left(\gamma\left(\frac{0}{s-1}\right),\gamma\left(\frac{1}{s-1}\right),\ldots, \gamma\left(\frac{s-1}{s-1}\right)\right).$$ As in~\cite{CLOT}, we use the term ``sectional category'' in the reduced sense. In other words, $\TC_s(X)+1$ stands for the smallest number of open sets covering $X^s$ in each of which $e_s$ admits a section. We agree to set $\TC_1(X):=\cat(X)$, the Lusterenik-Schnirelmann category of $X$ (also taken in the reduced sense).

A standard estimate for the $s$-th topological complexity of a space $X$ is given by:
\begin{prop}[{\cite[Theorem~3.9]{bgrt}}]\label{ulbTCnrararar}
For a $c$-connected space $X$ having the homotopy type of a CW complex, 
$$\cl(X)\leq\cat(X)\leq \hdim(X)/(c+1)
\quad\mbox{and}\quad
\zcl_s(X)\leq\TC_s(X)\leq s\cdot \cat(X).$$
\end{prop}

The notation $\hdim(X)$ stands for the (cellular) homotopy dimension of $X$, i.e.~the minimal dimension of CW complexes having the homotopy type of $X$. On the other hand (and for our purposes), the cup-length of $X$, $\cl(X)$, and the $s$-th zero-divisor cup-length of $X$, $\zcl_s(X)$, are defined in purely cohomological terms\footnote{All cohomology groups in this paper are taken with $\mathbb{Z}$-coefficients.}. The former one is the largest non-negative integer $\ell$ such that there are positive-dimensional classes $c_j\in H^*(X)$, $1\leq j\leq \ell$, with non-zero cup-product $c_1\cdots c_\ell\in H^*(X)$. Likewise, $\zcl_s(X)$ is the largest non-negative integer $\ell$ such that there are classes $z_j\in H^*(X^s)$, each with trivial restriction under the iterated diagonal inclusion $\Delta_s\colon X \hookrightarrow X^s$, and so that the cup-product $z_1\cdots z_\ell\in H^*(X^s)$ is non-zero. Each such class $z_i$ is called an $s$-th zero-divisor for $X$. The ``zero-divisor'' terminology comes from the observation that the map induced in cohomology by $\Delta_s$ restricts to the $s$-fold tensor power $H^*(X)^{\otimes s}$ to yield the $s$-iterated cup-product.

\subsection{Abrams model}
\begin{defi}
Let $G$ be a graph thought of as a 1-dimensional cell complex, and let $n$ be a positive integer. The labelled configuration space of $n$ particles on $G$ is the subspace $\mathcal{C}^nG= \prod^nG-\Delta$ of the $n$-cartesian product $\prod^nG$, where $\Delta = \{(x_1,\ldots,x_n) \in \prod^nG : x_i=x_j, i \neq j \}$ is the fat diagonal. The unlabelled configuration space of $G$ on $n$ points, $U\mathcal{C}^nG$, is the quotient space of $\mathcal{C}^nG$ by the action of the symmetric group $\mathcal{S}_n$, where the action is given by permutation of the factors.
\end{defi}

Configuration spaces are non-compact, so they cannot have a finite cell complex structure. We use instead Abrams' discretized homotopy model, which has a finite cubical complex structure. Here we briefly review the definition and main properties.

\begin{defi}
Take the product cell structure in $\prod^nG$; a cell in $\prod^nG$ is a cartesian product $a=a_1\times\cdots\times a_n$, where each $a_i$ is an (open) cell in $G$. We use the slightly simpler notation $a=(a_1,\ldots,a_n)$. The discretized configuration space of $n$ particles on $G$, $\mathcal{D}^nG$, is the subcomplex of $\prod^nG$ obtained by removing all cells whose closure intersect $\Delta$. Thus, a typical cell in $\mathcal{D}^nG$ has the form $a=(a_1,\ldots,a_n)$, with the closure of each $a_i$ being disjoint from the closure of any other $a_j$ ($i\neq j$). The symmetric group $\mathcal{S}_n$ acts on $\mathcal{D}^nG$ by permuting factors. The action permutes in fact cells, and the quotient complex, denoted by $U\mathcal{D}^nG$, is the unordered discretized configuration space of $n$ particles on $G$.
\end{defi}

Discretized configuration spaces can be used as homotopy models for usual configuration spaces.

\begin{teo}[\cite{aa,MR2833585}]\label{1}
For any $n >1$ and any graph $G$ with at least $n$ vertices, the labelled (unlabelled) configuration space of $n$ particles on $G$ strong deformation retracts onto $\mathcal{D}^nG$ $(U\mathcal{D}^nG)$ provided the following two conditions hold:
\begin{enumerate}
\item Each path between distinct vertices of degree not equal to 2 passes through at least $n-1$ edges.
\item Every cycle (i.e.~a loop whose only repeated vertices are the initial and final ones) passes through at least $n+1$ edges.
\end{enumerate}
\end{teo}

A graph satisfying the two conditions in Theorem~\ref{1} is said to be $n$-sufficiently subdivided. In this work we assume all graphs to be $n$-sufficiently subdivided, unless explicitly noted otherwise.

\subsection{Discrete Morse theory} 
\begin{defi}
A discrete vector field $W$ on a regular cell complex $X$ is a collection of pairs of cells of the form $(\tau, \nu)$, where $\tau $ is an inmediate face of $\nu$, such that each cell of $X$ appears as an entry of at most one pair of $W$. A cell $\sigma$ of $X$ is called:
\begin{itemize}
\item critical provided it does not appear as an entry of any pair of $W$;
\item redundant provided there is a cell $\sigma'$ such that $(\sigma,\sigma')\in W$;
\item collapsible provided there is a cell $\sigma'$ such that $(\sigma',\sigma)\in W$.
\end{itemize}
\end{defi}

Note that any cell must be of one and only one of the three types above. For a redundant cell $\tau$ of $X$, we shall denote by $W(\tau )$ the cell of $X$ with $(\tau, W(\tau))\in W$. 
 \begin{defi}
Let $W$ be a discrete vector field on $X$. A sequence of $k$-cells, $\tau_1, \tau_2,\ldots,\tau_n$ satisfying $\tau_i \neq \tau_{i+1}$ for $ i =1,2,\ldots,n-1$ is called a $W$-path of length $n$ if, for each $i=1,\ldots,n-1$, $\tau_i$ is redundant and $\tau_{i+1}$ is a face of $W(\tau_i)$. The $W$-path is closed if $\tau_1=\tau_n$. We say that $W$ is a gradient vector field provided it does not admit closed $W$-paths.
\end{defi}

The central idea behind discrete Morse theory is that a discrete gradient vector field $W$ on a regular complex $X$ provides instructions for simplifying the cellular structure of the simplicial complex $X$ without altering its homotopy type. The simplest situation to keep in mind is that of a cellular collapse: If $(\tau,\sigma)$ is a pair of cells in $X$ with $\tau$ a free facet of $\sigma$, we can collapse~$\sigma$ to $\partial\sigma-\tau$ by ``pushing'' it along $\tau$. This does not change the homotopy type of $X$, and produces a new cellular structure where, essentially, $\tau$ and $\sigma$ have been removed. A pair of cells as above would typically belong to a discrete gradient field on $X$. The collapsing process can be performed even if $\tau$ is not a free facet of $\sigma$ as long as we suitably elongate other cells (and their boundaries) having $\tau$ as a facet. More generally, we can iterate the collapsing process with any given \emph{non-closed} $W$-path $\tau_1,\tau_2,\ldots,\tau_n$ that starts at a cell $\tau_1$ that is either a free facet of $W(\tau_1)$, or that, except for $W(\tau_1)$, is a facet only of critical cells (the latter ones would then have to be elongated during the iterated collapsing process). The formal details are standard, see~\cite{rf}, and yield:

\begin{teo}\label{ftdmt}
Let $X$ be a finite cell-complex with a gradient vector field $W$. Then $X$ has the homotopy type of a cell complex having $m_p$ cells of dimension $p$, where $m_p$ denotes the number of critical cells in $X$ of dimension $p$.
\end{teo}

For our purposes, it is slightly more convenient to think in terms of the formulation in~\cite[Propositions~2.2 and~2.3]{fs}. Start with a finite regular CW complex $X$ having a discrete gradient vector field $W$. Let $X'_n$ (respectively, $X''_n$) be the subcomplex of $X$ obtained by removing all redundant $n$-cells (and all critical $n$-cells) from the $n$-skeleton $X_n$ of $X$. 

\begin{teo}\label{deformacionfuerteespecial}
\begin{enumerate}
\item If $X$ has no critical cells of dimension greater than $k$, then $X$ strong deformation retracts to $X'_k$.
\item For any $n$, $X'_n$ is obtained from $X''_n$ by attaching as many $n$-cells as critical $n$-cells there are in $X$. Furthermore, if $n>0$, then $X''_{n}$ strong deformation retracts to $X'_{n-1}$.
\end{enumerate}
\end{teo}

The following fact, taken from~\cite{sc}, will be made explicit in the next section in the case of Abrams' discrete homotopy model for configuration spaces on graphs.
\begin{prop}\label{liftchettih}
Let $X$ be a finite cubical complex with finite covering map $\pi : Y \rightarrow X$. Take the lifted cubical complex structure on $Y$, and let $V$ be a discrete gradient vector field on $X$. Then there exists a discrete gradient vector field $W$ on $Y$ which is the lift of $V$.
\end{prop}

\section{Farley-Sabalka ordered gradient field}\label{SeccionFSGF}
Proposition~\ref{liftchettih} asserts that Farley-Sabalka discrete gradient vector field on $U\mathcal{D}^nG$ lifts to one on $\mathcal{D}^nG$. In this section we make explicit the construction, paralleling Farley-Sabalka's steps in the unordered case.

Let $T$ be a fixed maximal tree of $G$ with a fixed root (i.e.~a vertex of degree one) which we denote by $0$. Edges outside $T$ are called deleted edges. Define an operation on vertices $\wedge : V(G) \times V(G) \rightarrow V(G)$ by setting $v \wedge u =w$ if the intersection of the unique directed $0v$-path in $T$ with the unique directed $0u$-path in $T$ is the unique directed $0w$-path in $T$. Note that $u\wedge v=v\wedge u$ and $u\wedge u=u$ for all vertices $u$ and $v$.

Given a vertex $v$ different from $0$ and of degree $d(v)$, fix a total ordering of the edges in $G$ adjacent to $v$ by labeling each edge with a number between 0 and $d(v)-1$, assigning the number 0 to  the edge that lies on the unique $0v$-path in $T$. The single edge adjacent to $0$ is given the label~$1$. Define a function $g: V(G) \times V(G) \rightarrow \mathbb{Z}$ by setting $g(v,w)$ to be the label of the edge adjacent to $v$ that lies on the unique $vw$-path in $T$  if $v \neq w$, and $g(w,w)=0$.

Totally order the vertices of $G$ as follows. Let $u$ and $v$ be two vertices of $G$ with $u\neq v$, and set $w=v \wedge u$. Then $u<v$ if and only if $u=w$ or $g(w,u)<g(w,v)$ when $u \neq w$. Notice that if a vertex $w$ lies on the unique $0v$-path in $T$, then $w \leq v$. For practical purposes, it is convenient to picture such an order of the vertices of $G$ by numbering them in the order in which they are first encountered through the following walk along the tree $T$: Choose an embedding of $T$ in the plane, and start the walk at the root vertex 0. Walking is to be done in such a way that, at any given intersection, we take the (say) left-most branch, turning around when reaching a vertex of degree 1. In these terms, we will denote a given vertex by its associated ordinal number, recalling that the root vertex has already been denoted by 0. Edges of $G$ will then be written as \emph{ordered} pairs of non-negative integers (i.e.~vertices) $(u,v)$ with $u<v$. Note that the vertex ordering induces a corresponding total ordering on edges of $T$. Indeed, given a vertex $v$ different from $0$, we shall denote by $e(v)$ the edge in $T$ which is adjacent to $v$ and has label 0. This sets a one-to-one correspondence between edges in $T$ and vertices with a positive label, thus transferring the ordering of vertices to edges of $T$.

Let $a=(a_1,\ldots,a_n)$ be a cell of $\mathcal{D}^nG$, and assume that $a_i$ is a vertex. If the closure of $e(a_i)$ intersects the closure of some cell $a_j$ ($j\neq i$), then we say that $a_j$ blocks $a_i$ in $a$, and that $a_i$ is blocked in $a$ (note that in such a case the index $j$ is completely determined by the index $i$). On the other hand, if $a_i$ is unblocked in $a$, we define the cell $\er_i(a):=(a_1,\ldots,a_{i-1},e(a_i),a_{i+1,}\ldots,a_n)$ of $\mathcal{D}^nG$ to be the elementary reduction of $a$ from $a_i$. If $a_i$ is the smallest unblocked vertex of $a$ (in the sense of the order on vertices), then we use the name ``principal reduction'' instead of ``elementary reduction'', and the notation $\pr(a)$ instead of $\er_i(a)$.

We next define inductively a collection $W$ of pairs of cells of $\mathcal{D}^nG$, which we shall later prove to form a discrete gradient field. Set $$W_0=\{(a,\pr(a))\,\vert\,a \mbox{ is a vertex of }\mathcal{D}^nG\mbox{ and $\pr(a)$ is defined}\}.$$ Assuming $W_k$ has been defined, let $W_{k+1}$ be the set of pairs $(a,\pr(a))$ such that $a$ is a $(k+1)$-cell of $\mathcal{D}^nG$ with $\pr(a)$ defined, and such that $a$ does not appear as the second entry of some pair in $W_k$. We then set $W=\bigcup_{k\geq0} W_k$.

 Let $(a,b)$ be an edge of $T$ (so that $a$ and $b$ are vertices with $a<b$ and $(a,b)=e(b)$). Let $c$ be a cell of $\mathcal{D}^nG$ containing the edge $(a,b)$. We say that $(a,b)$ is an order respecting edge in $c$ if for every vertex $c_i$ of $c$ adjacent to $a$, we have either $c_i < a$ or $c_i>b$. Equivalently, $(a,b)$ is an order respecting edge in $c$ if $b$ is smaller than any other vertex in $c$ that is blocked by $(a,b)$.

\begin{lema}[Order-respecting edges lemma]\label{order}
\begin{enumerate}
\item Let $a=(a_1,a_2,\ldots,a_n)$ be a cell of $\mathcal{D}^nG$ with $a_i$ the smallest unblocked vertex of $a$. Then $e(a_i)$ is an order respecting edge in $\pr(a)$.
\item Let $a=(a_1,a_2,\ldots,a_n)$ be a cell of $\mathcal{D}^nG$ with $a_i = e(b_i)$ for a vertex $b_i$ of $G$. Consider the cell $b=(a_1,\ldots,a_{i-1},b_i,a_{i+1},\ldots,a_n)$ of $\mathcal{D}^nG$ ---so $a=\er_i(b)$. Assume that a cell $a_j$ with $j\neq i$ is an edge of $T$. Then $a_j$ is order respecting in $a$ if and only if it is order respecting in $b$.
\end{enumerate}
\end{lema}
\begin{proof}
Regarding the first assertion, assume that the edge $e(a_i)=(v_i,a_i)$, $v_i <a_i$, is not an order respecting edge in $\pr(a)$. Then there exists a vertex $a_j$ in $\pr(a)$ adjacent to $v_i$ such that $v_i < a_j < a_i$. But this means that $a_i$ is not the minimal unblocked vertex of $a$, for $a_j$ is also unblocked in $a$ and $a_j <a_i$.

For the second assertion, if $a_j$ is order respecting in $b$, then it is necessarily order respecting in $a$, for any vertex in $a$ is also a vertex in $b$. Conversely, let $a_j=(x,y)$, $x<y$, $j\neq i$, be an order respecting edge in $a$, and assume it is not an order respecting edge in $b$. Then the vertex $b_i$ must be adjacent to the vertex $x$ with $x < b_i <y$. This forces $(x,b_i)=e(b_i)=a_i$, which contradicts that $a$ is a cell in $\mathcal{D}^nG$, for the closures of $a_i$ and $a_j$ overlap.
\end{proof}

In what follows we consider minimal order respecting edges in cells $a$ of $\mathcal{D}^nG$, where minimality is taken within order respecting edges in $a$ with respect to the order we have set on edges of $T$.

\begin{teo}[Classification Theorem]\label{teoremadeclasificacion}
A cell $a=(a_1,\ldots,a_n)$ in $\mathcal{D}^nG$ is:
\begin{enumerate}
\item critical if and only if it contains no order respecting edges nor unblocked vertices.
\item collapsible if and only if it contains a minimal order respecting edge $a_i=e(v)$, but no unblocked vertices strictly smaller than $v$.
\item redundant if and only if either
\begin{enumerate}
\item it does not contain order respecting edges, {but it does contain unblocked vertices,} or
\item it contains a minimal order respecting edge $a_i=e(v)$ as well as an unblocked vertex~$a_j$ strictly smaller than $v$.
\end{enumerate}
\end{enumerate}
\end{teo}

Item 3 in Theorem~\ref{teoremadeclasificacion} can be stated on the lines of item 2. Namely, $a$ is redundant if and only if it contains a minimal unbloqued vertex $a_j$, but no order respecting edges strictly smaller than $a_j$. We have chosen the statement using subitems \emph{(a)} and \emph{(b)} for proof purposes.

\begin{proof}[Proof of Theorem~\ref{teoremadeclasificacion}]
We start by proving sufficiency of all four conditions.

\smallskip
1. Assume that $a$ contains no order respecting edges nor unblocked vertices. Then $a$ cannot be collapsible (by Lemma~\ref{order}.1) nor redundant (by definition), so it is critical. (Note that the last conclusion holds even though we have not yet proved that $W$ is a discrete vector field.)

\smallskip
2. Assume now that $a$ contains a minimal order respecting edge $a_i=e(v)$, but no unblocked vertices smaller than $v$. Set $b=(a_1,\ldots,a_{i-1},v,a_{i+1},\ldots,a_n)$. We prove that $(b,a)\in W$ (so $a$ is collapsible, as asserted). Let $u$ be the smallest unblocked vertex of $b$. Then $u \leq v$ since $v$ is also an unblocked vertex of $b$. If $u<v$, then $u$ must be blocked in $a$, by the second hypothesis. But $u$ is unblocked in $b$, so that $u$ must be blocked precisely by $a_i$ in $a$, contradicting the first hypothesis. Hence $u=v$, so that $a=\pr(b)$. The only thing left to prove is that $b$ is not collapsible. Suppose, for a contradiction, that $b$ is collapsible. By Lemma~\ref{order}.1 $b$ has an order respecting edge $a_j=e(w)$, with $j\neq i$, such that \begin{equation}\label{dosigualdades}b=\er_j(c)=\pr(c),\end{equation} where $c$ is obtained from~$b$ by substituting the edge $a_j$ by the vertex $w$. Note that $v$ is necessarily unblocked in $c$, so that the second equality in~(\ref{dosigualdades}) (together with the fact that $j\neq i$) yields \begin{equation}\label{contrad1}w<v.\end{equation} On the other hand, recall that $a_j$ is order respecting in $b$, so Lemma~\ref{order}.2 implies that $a_j$ is order respecting in $a$ too. The assumed minimality of $a_i$ then yields $v\leq w$, which contradicts~(\ref{contrad1}).

\smallskip
3(a). Assume next that $a$ does not contain order respecting edges, but it does contain unblocked vertices. Then $a$ cannot be collapsible (by Lemma~\ref{order}.1), and is in fact redundant (by definition).

\smallskip
3(b). Lastly, assume that $a$ contains a minimal order respecting edge $a_i=e(v)$, as well as an unblocked vertex $a_j$ with
\begin{equation}\label{e2}
a_j < v.
\end{equation}
We can safely assume that $a_j$ is the minimal such vertex in~$a$. Once we prove that $a$ is not collapsible, it will be clear that $a$ is redundant, for $\er_j(a)=\pr(a)$ is defined. Assume, for a contradiction, that $a$ is collapsible. By Lemma~\ref{order}.1, there exists an order respecting edge $a_k=e(w)$ of $a$ giving
\begin{equation}\label{e4}
a=\er_k(b)=\pr(b),
\end{equation}
where $b$ is obtained from $a$ by substituting the edge $a_k$ by the vertex $w$. The assumed minimality of $a_i$ gives
\begin{equation}\label{e3}
v\leq w.
\end{equation}
On the other hand, the vertex $a_j$ has been assumed to be unblocked in $a$, so it clearly belongs to and is unblocked in $b$. It then follows from~(\ref{e4}) that $w\leq a_j$, which contradicts~(\ref{e2}) and~(\ref{e3}).

Since the four conditions in this theorem are mutually disjoint and cover all possible situations, the proof is complete in view of Theorem~\ref{teo1} below.
\end{proof}

\begin{teo}\label{teo1}
$W$ is a discrete vector field in $\mathcal{D}^nG$.
\end{teo}

From the way $W$ has been defined it is clear that (i) no cell in $\mathcal{D}^nG$ is simultaneously redundant and collapsible, and that (ii) no cell $a$ in $\mathcal{D}^nG$ belongs to two different pairs $(a,b)$ of $W$. Consequently, the proof of Theorem~\ref{teo1} amounts to proving that $W$ is an injective function, a fact that we will prove as a consequence of the following strengthening of Lemma~\ref{order}.1:

\begin{prop}\label{Wisinjective}
Let $a=(a_1,\ldots,a_n)$ be a collapsible cell in $\mathcal{D}^nG$, say $a=\er_i(b)=W(b)$. Then $a_i$ is the minimal order respecting edge of $a$.
\end{prop}

The proof of Proposition~\ref{Wisinjective} uses the sufficiency of all four conditions in Theorem~\ref{teoremadeclasificacion}. Since this sufficiency part has already been established, our arguments make no  loop-type logical mistake.

\begin{proof}[Proof of Proposition~\ref{Wisinjective}]
By Lemma~\ref{order}.1, $a_i$ is an order respecting edge in $a$. Say $a_i=e(v)$ and let $a_j=e(u)$ be the minimal order respecting edge of $a$, so that $u\leq v$. Assume, for a contradiction, that $u<v$ (so $a_j\neq a_i$ and, in particular, $j\neq i$). The edge $a_j$ is order respecting in $b$ (by Lemma~\ref{order}.2, as $a_j$ is an order respecting edge in $a$ ---note this uses the assumption $j\neq i$). Since $b$ is redundant (so non-collapsible), the sufficiency part in Theorem~\ref{teoremadeclasificacion}.2 implies that $b$ has an unblocked vertex $a_k$ such that $a_k<u$. In turn, the last inequality, and the sufficiency part in Theorem~\ref{teoremadeclasificacion}.3(b) imply that $a_k$ is blocked in $a$ (and, as note above, unblocked in $b$). Thus $a_k$ is blocked in fact by $a_i$. But $a_i=e(v)$ is order respecting in $a$, so $a_k>v$ is forced, in contradiction to the previously drawn inequalities.
\end{proof}

\begin{proof}[Proof of Theorem~\ref{teo1}]
Let $a=(a_1,\ldots,a_n)$ and $a'=(a'_1,\ldots,a'_n)$ be redundant cells in $\mathcal{D}^nG$ with $W(a)=W(a')$. Let $a_i$ and $a'_{i'}$ be the corresponding minimal unblocked vertices. By Proposition~\ref{Wisinjective}, the minimal order respecting edge in
$$
(a_1,\ldots,a_{i-1},e(a_i),a_{i+1},\ldots,a_n)=(a'_1,\ldots,a'_{i'-1},e(a'_{i'}),a'_{i'+1},\ldots,a'_n)
$$
is $e(a_i)=e(a'_{i'})$ with $i=i'$. This yields $a_i=a'_{i'}$ and, in fact, $a=a'$.
\end{proof}

By construction, $W$ is equivariant with respect to the standard action of the symmetric group $\mathcal{S}_n$ on $\mathcal{D}^nG$, i.e., $a=W(b)\Rightarrow \sigma\cdot a=W(\sigma\cdot b)$, for $\sigma\in\mathcal{S}_n$. Furthermore, $W$ yields Farley-Sabalka's discrete (gradient) vector field in the quotient $U\mathcal{D}^nG$. Consequently, the following facts follow directly from their unordered analogues in~\cite{fs}\footnote{See~\cite{MR2537376}, particularly Theorem~4.2, for the $\mathcal{S}_n$-equivariant ingredient in the third item of Corollary~\ref{dim}.}:

\begin{coro}\label{dim}\begin{enumerate}
\item $W$ is a discrete gradient vector field.
\item If $U\mathcal{D}^nG$ has $m_k$ critical cells of dimension $k$, then $\mathcal{D}^nG$ has $n!m_k$ critical cells of dimension $k$.
\item $\mathcal{D}^nG$ strong and $\mathcal{S}_n$-equivariantly deformation retracts to a CW complex of dimension at most $k(n,G)$, where
$$k(n,G)= \min\left\{ \,\left\lfloor \frac{n+1-\chi(G)}{2} \right\rfloor \hspace{.5mm}, \hspace{1mm}\mbox{\emph{card}} \rule{-.5mm}{4mm}\left\{ \rule{0mm}{3.5mm}v\in V(G): d(v) \geq 3\right\} \,\right\}.$$
Indeed, a maximal tree $T$ in $G$ can be chosen so that its associated gradient vector field has critical cells of dimension at most $k(n,G)$.
\end{enumerate}\end{coro}

\begin{remark}\label{hshiausy}{\em
When $G$ is a tree, the maximal tree $T$ in Corollary~\ref{dim}.3 is forced to be $G$ itself, and $k(n,G)$ becomes
$$k(n,G)= \min\left\{ \,\left\lfloor \frac{n}{2} \right\rfloor \hspace{.5mm}, \hspace{1mm}\mbox{\emph{card}} \rule{-.5mm}{4mm}\left\{ \rule{0mm}{3.5mm}v\in V(G): d(v) \geq 3\right\} \,\right\}.$$
}\end{remark}

\section{Combinatorics in the 1-skeleton of the Morse model}\label{secciondecombinatoria}
In this section $G=T$ is a fixed tree with at least one essential vertex. For $r\geq1$, set $$m_r=m_r(T)=\frac{1}{r!}\sum(d(v)-1)(d(v)-2)\cdots(d(v)-r)=\sum\binom{d(v)-1}{r},$$
where both summations run over the vertices $v$ of $T$ with $d(v)>r$, and let $B_r$ stand for the $r$-banana graph, i.e., the graph with two vertices and $r$ edges, none of which is a loop.

\smallskip
As a warmup, and for future reference, we use Farley-Sabalka's model to reprove Theorem~11.1 in~\cite{mp}:
\begin{teo}\label{farbergeneralized}
If $T$ is 2-sufficiently subdivided, then $\mathcal{D}^2T$ has the $\mathcal{S}_2$-equivariant homotopy type of the banana graph $B_{2m_2}$.
\end{teo}

\begin{remark}\label{clearremark}{\em
As it will be clear from the proof of Theorem~\ref{farbergeneralized}, the generator $\sigma\in\mathcal{S}_2$ interchanges the two vertices $(1,0)$ and $(0,1)$ of $B_{2m_2}$. The full effect of the (continuous and fix-point free) map $\sigma\colon B_{2m_2}\to B_{2m_2}$ is then forced on us: $\sigma$ induces a pairing $P_\sigma$ on the edges of $B_{2m_2}$ in such a way that each edge~$e$, say oriented from $(1,0)$ to $(0,1)$, is sent homeomorphically under $\sigma$ into the $\sigma$-paired edge $P_\sigma(e)$ oriented backwards, i.e., from $(0,1)$ to $(1,0)$ (cf.~\cite[Section~11]{mp} or, more generally, Theorem~4.2 in~\cite{MR2537376}). As shown later in the section, such an ``$\mathcal{S}_2$-equivariant banana-type'' phenomenon in $\mathcal{D}^2T$ propagates to other $\mathcal{D}^nT$ ($n\geq3$), where we find in fact more general and nicely organized ``$\mathcal{S}_n$-equivariant multi-edge cyclic-graph'' phenomena.
}\end{remark}
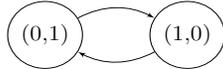
\begin{figure}[ht!]
	\centering
\resizebox{3cm}{.9cm}{
\begin{tikzpicture}
\tikzset{vertex/.style = {shape=circle,draw,minimum size=1.5em}}
\tikzset{edge/.style = {->,> = latex'}}
		\node[vertex] (1) at ( 2.3, 3.8) {(0,1)};
		\node[vertex] (2) at ( 4.8, 3.8) {(1,0)};
		\draw[edge] (2) to [bend left] (1);
                   \draw[edge] (1) to [bend left] (2);
	\end{tikzpicture}}
	\caption{Graph onto which $(\mathcal{D}^2T)'_1$ strong and $\mathcal{S}_2$-equivariantly deformation retracts. Each of the two arrows represents $m_2$ edges, with arrow direction standing for the $\mathcal{S}_2$-action.}
	\label{dcero}
\end{figure}
\begin{proof}[Proof of Theorem~\ref{farbergeneralized}]
By the Classification Theorem~\ref{teoremadeclasificacion}, $\mathcal{D}^2T$ has only two critical vertices, namely $(1,0)$ and $(0,1)$, and no critical cells of dimension greater than 1, so $\mathcal{D}^2T$ has the $\mathcal{S}_2$-equivariant homotopy type of a graph with two vertices. We take a closer look at the structure of this graph, showing first that it has no loops (so it is connected), i.e.~it is a banana graph, and then we count its edges.

We start from $(\mathcal{D}^2T)'_1$, which is a strong and $\mathcal{S}_2$-equivariant deformation retract of $\mathcal{D}^2T$ in view of Theorem~\ref{deformacionfuerteespecial} and Corollary~\ref{dim}. As explained in the paragraph preceding Theorem~\ref{ftdmt}, the graph we want is obtained from $(\mathcal{D}^2T)'_1$ by further collapsing all $W$-collapsible 1-cells. Consequently, in order to see that the resulting graph is banana-type, all we need to check is that the (necessarily) unique gradient paths from the endpoints of any given critical edge in $\mathcal{D}^2T$ land in different critical vertices.

By the Classification Theorem~\ref{teoremadeclasificacion}, a critical 1-cell of $\mathcal{D}^2T$ is either of the form $((a,c),b)$ or $(b,(a,c))$, with $b$ adjacent to $a$ and $a<b<c$. In the first case, the facets of $((a,c),b)$ are $(a,b)$ and $(c,b)$, and it is transparent from the definitions that the asserted unique $W$-paths are
$$(a,b)=(a_k,b),(a_{k-1},b),\ldots,(a_0,b)=(0,b),(0,a)=(0,a_k),(0,a_{k-1}),\ldots,(0,a_1)=(0,1)$$ and 
$$(c,b),(c,a)=(c,a_k),(c,a_{k-1}),\ldots,(c,a_0)=(c,0),(a,0)=(a_k,0),(a_{k-1},0),\ldots,(a_1,0)=(1,0),$$ 
where $0=a_0,1=a_1,a_2,\ldots,a_k=a$ is the sequence of vertices encountered in the unique directed $0a$-path in $T$. The case of a critical cell $(b,(a,c))$ is completely parallel, and verifiable either by direct analysis, as above, or as a $\mathcal{S}_2$-equivariance consequence of the first case.

It remains to count the critical edges of $\mathcal{D}^2T$ or, equivalently, the critical edges of $U\mathcal{D}^2T$. Recall the function $g\colon V(T)\times V(T)\to\mathbb{Z}$ defined at the beginning of Section~\ref{SeccionFSGF}. We already noted that the critical edges of $U\mathcal{D}^2T$ are those of the form $\{b,(a,c)\}$, with $b$ adjacent to $a$, and $a<b<c$. In particular $a$ is an essential vertex. Since $0<g(a,b)<g(a,c)$, the number of critical 1-cells of $U\mathcal{D}^2T$ involving a fixed essential vertex $a$ is then
$$
1+2+\cdots+(d(a)-2)=\frac{(d(a)-1)(d(a) -2 )}2
$$
(for instance, the summand $d(a)-2$ on the left hand-side of the previous equality accounts for instances with $g(a,c)=d(a)-1$.) Consequently, the number of critical 1-cells of $U\mathcal{D}^2T$ is $m_2$, and the result follows from Corollary~\ref{dim}.2.
\end{proof}

\begin{remark}\label{inclusion}{\em
Let $T'$ be a subtree of $T$ with compatible orderings of vertices and edges. As explained in the proof of~\cite[Corollary~11.3]{mp}, the compatibility condition can be attained by restricting an embedding of $T$ in the plane to one for $T'$, and choosing corresponding roots $0'$ and $0$ for $T'$ and $T$ so that the $T$-path connecting $0$ to $0'$ is disjoint from $T'-\{0'\}$. It follows from the previous proof (keep in mind the explanation given in the paragraph above Theorem~\ref{ftdmt}) that $\mathcal{S}_2$-equivariant homotopy equivalences $\mathcal{D}^2T'\simeq B'$ and $\mathcal{D}^2T\simeq B$ can be constructed in such a way that the obvious embedding $\mathcal{D}^2T'\hookrightarrow\mathcal{D}^2T$ corresponds, up to homotopy, to a $\mathcal{S}_2$-equivariant subgraph inclusion $B'\hookrightarrow B$ between the corresponding banana-graph models. For instance, if $T'$ is a $Y$-graph embedded around a small neighborhood of an essential vertex in $T$ (with compatible orderings, as above), then $B'$ embeds in $B$ as the circle formed by the $P_\sigma$-pair of critical edges determined by $T'$ (see Remark~\ref{clearremark}). 
}\end{remark}

We use the same technique in order to study configuration spaces with more particles. We treat first the case of three particles, as it is still special ---see the sentence containing~(\ref{noexcepcional}) in the next section--- and accessible through explicit pictures, in addition to shedding good light on the type of phenomena we find for configurations with more particles.
\begin{teo} \label{n3}
If $T$ is 3-sufficiently subdivided, then $\mathcal{D}^3T$ has the $\mathcal{S}_3$-equivariant homotopy type of the graph $G(3,T)$ obtained by splicing the graphs depicted in Figures~\ref{3} and~\ref{d3}. (The $\mathcal{S}_3$-action on $G(3,T)$ is described along the proof below.)
\end{teo}
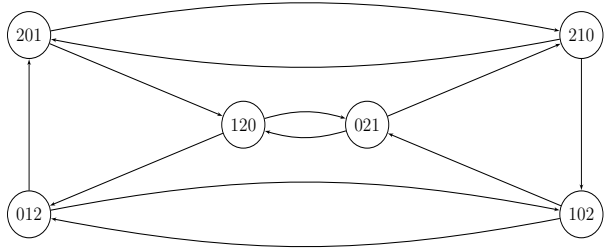
\begin{figure}[ht!]
\begin{center}
\resizebox{8cm}{3.5cm}{
\begin{tikzpicture}
\tikzset{vertex/.style = {shape=circle,draw,minimum size=1.5em}}
\tikzset{edge/.style = {->,> = latex'}}
\node[vertex] (a) at  ( 1, 1) {$012$};
\node[vertex] (b) at ( 1, 5) {$201$};
\node[vertex] (c) at  ( 14.38, 1) {$102$};
\node[vertex] (d) at  ( 14.38, 5) {$210$};
\node[vertex] (a1) at ( 6.19, 3) {$120$};
\node[vertex] (a2) at ( 9.19, 3) {$021$};
\draw[edge] (a) -- (b);
\draw[edge] (c) to (a2);
\draw[edge] (a2) to (d);
\draw[edge] (d) to (c);
\draw[edge] (a1)  to (a);
\draw[edge] (b) to (a1);
\draw[edge] (a1) to [out=15,in=165] (a2);
\draw[edge] (a2) to [out=195,in=-15] (a1);
\draw[edge] (a) to [out=10,in=170] (c);
\draw[edge] (c) to [out=190,in=-10] (a);
\draw[edge] (b) to [out=10,in=170] (d);
\draw[edge] (d) to [out=190,in=-10] (b);
\end{tikzpicture}}
\end{center}\caption{First half of the edge structure of $G(3,T)$. Each arrow represents $m_2+m_3$ edges.}\label{3}
\end{figure}

\begin{figure}
	\centering
\resizebox{2.5cm}{2.7cm}{
\begin{tikzpicture}
\tikzset{vertex/.style = {shape=circle,draw,minimum size=1.5em}}
\tikzset{edge/.style = {->,> = latex'}}
		\node[vertex] (1) at ( 2.3, 4.2) {201};
		\node[vertex] (2) at ( 4.8, 4.2) {102};
		\node[vertex] (3) at ( 2.3, 5.8) {012};
		\node[vertex] (4) at ( 4.8, 5.8) {021};
		\node[vertex] (5) at ( 2.3, 7.4) {120} ;
		\node[vertex] (6) at ( 4.8, 7.4) {210};
		\draw[edge] (2) to [out=200,in=-20] (1);
		\draw[edge] (4) to [out=200,in=-20] (3);
		\draw[edge] (6) to [out=200,in=-20] (5);
                   \draw[edge] (1) to [out=20,in=160] (2);
		\draw[edge] (3) to [out=20,in=160] (4);
		\draw[edge] (5) to [out=20,in=160] (6);
	\end{tikzpicture}}
	\caption{Second half of the edge structure of $G(3,T)$. Each arrow represents $m_2$ edges.}
	\label{d3}
\end{figure}
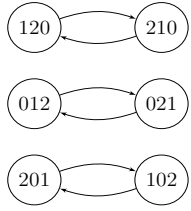

\begin{proof}
By the Classification Theorem~\ref{teoremadeclasificacion}, $\mathcal{D}^3T$ has six critical vertices, namely, the six permutations of $(0,1,2)$. (We simplify the notation $(i,j,k)$ to $ijk$ in Figures~\ref{3} and~\ref{d3}.) As in Theorem~\ref{farbergeneralized}, the graph $G(3,T)$ we want is obtained from $(\mathcal{D}^3T)'_1$ by collapsing all $W$-collapsible 1-cells. Take a critical edge $\{(a,c),b,d\}$ of $U\mathcal{D}^3T$, with $a<b<c$ and $b$ adjacent to $a$, and fix the unique representative $\varepsilon=((a,c),b,d)$ in $\mathcal{D}^3T$ which satisfies in addition the condition $b<d$ whenever $d$ is not blocked by $c$ or by the root vertex 0 (as in the first three instances in Figure~\ref{figura2}). The faces of $\varepsilon$ are $(a,b,d)$ and $(c,b,d)$. We shall analyze the (again unique) gradient path from each of these faces to a critical vertex; the corresponding situation for other representatives will then follow from the $\mathcal{S}_3$-equivariance. We consider all possible cases depending on how $d$ is blocked in~$\varepsilon$ (see Figure~\ref{figura2}).

\begin{figure}[ht!]
\begin{center}
	\begin{tikzpicture}[every node/.style={circle, draw, scale=.6}, scale=1.0, rotate = 180, xscale = -1]
                  \node (21) at ( -1.2, 2.3) {b};
		\node (22) at ( -.5, 2) {d};
		\node (23) at ( .2, 2.3) {c};
		\node (24) at ( -.5, 3) {a};
		\node (25) at ( -.5, 4) {0};
		
		\node (1) at ( 1.5, 2.3) {b};
		\node (2) at ( 3.3, 3) {d};
		\node (3) at ( 2.9, 2.3) {c};
		\node (4) at ( 2.2, 3) {a};
		\node (5) at ( 2.2, 4) {0};
		
		\node (6) at ( 4.3, 1.6) {d};
		\node (7) at ( 5.7, 3) {a};
		\node (8) at ( 5, 2.3) {b};
		\node (9) at ( 6.4, 2.3) {c};
		\node (10) at (5.7, 4) {0};
		
		\node (11) at ( 7.7, 2.3) {b};
		\node (12) at ( 8.4, 3) {a};
		\node (13) at ( 9.1, 2.3) {c};
		\node (14) at ( 9.9, 1.6) {d};
		\node (15) at ( 8.4, 4) {0};
		
		\node (16) at ( 11.2, 2.3) {b};
		\node (17) at ( 12.6, 2.3) {c};
		\node (18) at ( 11.9, 3) {a};
		\node (19) at ( 11.9, 4) {0};
		
                  \draw (24) -- (21);
		\draw (24) -- (22);
		\draw (24) -- (23);
		\draw[dashed,-] (25) -- (24);
		\draw (4) -- (1);
		\draw (4) -- (2);
		\draw (4) -- (3);
		\draw[dashed,-] (5) -- (4);
		\draw (8) -- (6);
		\draw (7) -- (8);
		\draw (9) -- (7);
		\draw[dashed,-] (10) -- (7);
		\draw (12) -- (11);
		\draw[dashed,-] (15) -- (12);
		\draw (13) -- (12);
		\draw (14) -- (13);
		\draw (18) -- (16);
		\draw (17) -- (18);
		\draw[dashed,-] (19) -- (18);
	\end{tikzpicture}
\end{center} \caption{Vertex $d$ can be blocked in $\varepsilon$ either by $a$ (first two diagrams), by $b$, by $c$, or else because it is the root vertex $0$. Dashed edges represent the unique path in $T$ connecting $a$ to the root vertex. The five pictures correspond to the situations treated in the five cases of the proof.}\label{figura2}
\end{figure}
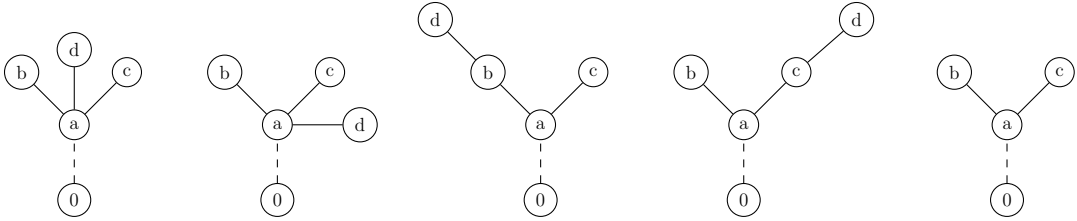

\noindent \textbf{Case 1:} (Just as $b$,) $d$ is blocked in $\varepsilon$ by $a$ with $a<d<c$. Recall we have chosen the representative $((a,c),b,d)$ in $\mathcal{D}^3T$ having $b<d$ (first diagram in Figure~\ref{figura2}). Then, as in the proof of Theorem~\ref{farbergeneralized}, it is easy to see that the unique gradient path starting at $(a,b,d)$ ends at $(0,1,2)$, and that the unique gradient path starting at $(c,b,d)$ ends at $(2,0,1)$. We write the latter facts by means of the short hand 
$$
\mbox{$(a,b,d)\stackrel{\mbox{\tiny grad }}{\longrightarrow}(0,1,2)$ \ \ and \ \ $(c,b,d)\stackrel{\mbox{\tiny grad }}{\longrightarrow}(2,0,1)$.}
$$
As explained in the paragraph preceding Theorem~\ref{ftdmt}, this means that, after performing the $W$-collapsing instructions, we get a family of edges in the Morse model connecting the vertices $(0,1,2)$ and $(2,0,1)$. These edges have a canonical orientation, from $(0,1,2)$ to $(2,0,1)$, which comes from the canonical orientation of the critical edge $((a,c),b,d)$ from its ``initial'' face $(a,b,d)$ to its ``final'' face $(c,b,d)$ ---recall $a<c$. Further, the cardinality (or ``multiplicity'', as we will call it below) of the resulting family of edges is the number of critical edges $((a,c),b,d)$ in this case (the counting is done below). As the reader can easily check, taking into account the $\mathcal{S}_3$-equivariance, we are led to the structure of six straight arrows in Figure~\ref{3}. For instance, under the action of the transposition $(1,2)\in\mathcal{S}_3$, each of the Morse edges from $(0,1,2)$ to $(2,0,1)$ becomes a corresponding Morse edge from $(1,0,2)$ to $(0,2,1)$. Regarding multiplicity, i.e., the number of possible critical edges $((a,c),b,d)$ fitting in this case, we note that, by fixing $a$, there are
$$
\binom22+\binom32+\cdots+\binom{d(a)-2}2=\frac{(d(a)-3)(d(a)-2)(d(a)-1)}{3!}
$$
such possibilities (for instance, the first summand $\binom22$ above counts the possibilities with $g(a,c)=3$). So we have a total of $m_3$ possibilities after letting $a$ vary over the essential vertices. Lastly, regarding the $\mathcal{S}_3$-action on $\mathcal{D}^3T$ (or, equivalently, the monodromy of the covering $\mathcal{D}^3T\to U\mathcal{D}^3T$), the orientation we have chosen for the unordered critical edge $\{(a,c),b,d\}$ and for its ordered representatives yields the indicated orientations for the batch of straight arrows in Figure~\ref{3}: triangles in that figure get oriented clockwise. For instance, the cycle $\sigma=(1,2,3)\in\mathcal{S}_3$ acts on the triangle on the left (right) of Figure~\ref{3} as a 120-degrees clockwise (counterclockwise) rotation. 

\smallskip
\noindent \textbf{Case 2:} $d$ is  blocked in $\varepsilon$ by $a$, but $a<b<c<d$ (second diagram in Figure~\ref{figura2}). We have
$$
\mbox{$(a,b,d)\stackrel{\mbox{\tiny grad}}{\longrightarrow}(0,1,2)$ \ \ and \ \ $(c,b,d)\stackrel{\mbox{\tiny grad}}{\longrightarrow}(1,0,2)$,}
$$
and the $\mathcal{S}_3$-equivariance now leads to the structure of six bent arrows in Figure~\ref{3}. Again, the multiplicity of each such edge is the number of possible critical edges $((a,c),b,d)$ fitting in this case. Fixing $a$, there are
$$
1\cdot(d(a)-3)+2\cdot(d(a)-4)+\cdots+(d(a)-3)\cdot1=\frac{(d(a)-3)(d(a)-2)(d(a)-1)}{3!}
$$
such possibilities (for instance, the first summand above, i.e.~$1\cdot(d(a)-3)$, counts the possibilities with $g(a,c)=2$; this counting strategy is used repeatedly below without further notice). So, as in the previous case, we have a total of $m_3$ possibilities after letting $a$ vary over the essential vertices. Regarding the corresponding monodromy, note that the batch of bent arrows in Figure~\ref{3} is oriented in such a way that ``twin'' arrows have opposite orientations.

\smallskip
\noindent \textbf{Case 3:} $d$ is blocked by $b$ (third diagram in Figure~\ref{figura2}). Since $T$ is 3-sufficiently subdivided, in fact $a<b<d=b+1<c$. This time we have
$$
\mbox{$(a,b,d)\stackrel{\mbox{\tiny grad}}{\longrightarrow}(0,1,2)$ \ \ and \ \ $(c,b,d)\stackrel{\mbox{\tiny grad}}{\longrightarrow}(2,0,1)$}
$$
and, as in Case 1, the $\mathcal{S}_3$-equivariance leads to structure of the six straight arrows in Figure~\ref{3}, with the exact same monodromy. The difference being on the multiplicity of this second batch of straight arrows: The number of possible critical edges $((a,c),b,d)$ fitting in this case with a fixed essential vertex $a$ is
$$
1+2+\cdots+(d(a)-2)=\frac{(d(a)-2)(d(a)-1)}{2},
$$
so we have a total of $m_2$ possibilities after letting $a$ vary over the essential vertices.

\smallskip
\noindent \textbf{Case 4:} $d$ is blocked by $c$ (fourth diagram in Figure~\ref{figura2}). Now we actually have $a<b<c<d=c+1$, and
$$
\mbox{$(a,b,d)\stackrel{\mbox{\tiny grad}}{\longrightarrow}(0,1,2)$ \ \ and \ \ $(c,b,d)\stackrel{\mbox{\tiny grad}}{\longrightarrow}(1,0,2)$.}
$$ So, as in Case 2 above, the $\mathcal{S}_3$-equivariance leads to the structure of six bent arrows in Figure~\ref{3}, with the exact same monodromy. The multiplicity is $m_2$, by the exact same counting in Case~3. 

\smallskip
\noindent \textbf{Case 5:} $d=0$ (last diagram in Figure~\ref{figura2}). We have $d=0<a<b<c$, and
$$
\mbox{$(a,b,d)\stackrel{\mbox{\tiny grad}}{\longrightarrow}(1,2,0)$ \ \ and \ \ $(c,b,d)\stackrel{\mbox{\tiny grad}}{\longrightarrow}(2,1,0)$,}
$$ which leads, by $\mathcal{S}_3$-equivariance, to the structure of six bent arrows in Figure~\ref{d3}, with multiplicity~$m_2$ (by the exact same counting as that in the previous two cases), and monodromy such that ``twin'' arrows have opposite orientations.
\end{proof}

\subsection{The general case $\mathcal{D}^nT$.}
The relevant homotopy information in the 1-skeleton of the Morse model for $\mathcal{D}^nT$ can be addressed by conveniently organizing the case-by-case analysis in the proof above. This leads us to a graph with a structure of edges by ``levels''. For instance, in the case of $\mathcal{D}^2T$ (Farber's case), there is a single edge level (the banana graph in Figure~\ref{dcero}), while two levels (shown in Figures~\ref{3} and~\ref{d3}) arise in the case of $\mathcal{D}^3T$. Note that there are only banana-graph components in the level shown by Figure~\ref{d3}. But, in addition to banana (sub)graphs, we also find multi-edge 3-cycles in the ``higher'' level shown by Figure~\ref{3}. As we will see below, multi-edge $c$-cycles with $c\leq n$ appear organized by levels in the case of $\mathcal{D}^nT$.

In full detail, let $G(n,T)$ be the graph obtained from $(\mathcal{D}^nT)_1^{'}$ after collapsing all $W$-collapsible 1-cells. Our interest in $G(n,T)$ comes from the fact that, since the difference between $(\mathcal{D}^nT)_1$ ---the complete 1-skeleton of $\mathcal{D}^nT$--- and $(\mathcal{D}^nT)_1^{'}$ consists of 1-cells that are redundant in $\mathcal{D}^nT$, the composite $G(n,T)\simeq (\mathcal{D}^nT)_1^{'}\hookrightarrow \mathcal{D}^nT$ induces an epimorphism at the level of fundamental groups. Our aim is to get full control on the generators of the fundamental group of $\mathcal{D}^nT$ (and thus of $\mathcal{C}^nT$) by understanding in full the combinatorial structure of $G(n,T)$.

As explained in the paragraph preceding Theorem~\ref{ftdmt}, there is a one-to-one correspondence between vertices (respectively, edges) of $G(n,T)$ and critical $0$-cells (respectively, $1$-cells) of $(\mathcal{D}^nT)'_1$. For instance, by the Classification Theorem~\ref{teoremadeclasificacion}, $(\mathcal{D}^nT)'_1$ has $n!$ critical $0$-cells, namely, the permutations of $(0,1,\ldots,n-1)$. The vertex set $V(G(n,T))$ of $G(n,T)$ will thus be identified with the underlying set of the symmetric group $\mathcal{S}_n$. Explicitly, a permutation $\sigma\in\mathcal{S}_n$ will stand for the vertex of $G(n,T)$ corresponding to the critical $0$-cell $(a_1,\ldots,a_n)$ of $(\mathcal{D}^nT)'_1$, where
\begin{equation}\label{0iden}
\sigma(i)=a_i+1 \mbox{ for } i\in\{1,2,\ldots,n\}.
\end{equation}
The corresponding identification between edges $e$ of $G(n,T)$ and critical 1-cells $c$ in $(\mathcal{D}^nT)_1^{'}$ will be spelled out by writing $e=e(c)$ and $c=c(e)$.

The standard (right) action of a permutation $\sigma\in\mathcal{S}_n$ on a configuration $(x_1,\ldots,x_n)\in\mathcal{C}^nT$ is
\begin{equation}\label{raccion}
(x_1,\ldots,x_n)\cdot\sigma=(x_{\sigma(1)},\ldots,x_{\sigma(n)}).
\end{equation}
This restricts to the (right) action of $\mathcal{S}_n$ on the discretized configuration space $\mathcal{D}^nT$. Taking into account~(\ref{0iden}), this means that the (right) action $ V(G(n,T))\times \mathcal{S}_n\to V(G(n,T))$ is given by product of permutations. In the rest of the section we (first) analyze and (then) prove the following full description of the edge set $E(G(n,T))$ of $G(n,T)$ and of the action $E(G(n,T))\times\mathcal{S}_n\to E(G(n,T))$:
\begin{teo}\label{elcasotegeneral}
If $T$ is $n$-sufficiently subdivided, then the edge set of $\hspace{.4mm}G(n,T)$ decomposes as a pairwise disjoint union 
$$
E(G(n,T))=\bigsqcup_{2\leq j\leq i\leq n}E_{i,j}\,
$$
where, for each $i$ and $j$ as above, $E_{i,j}$ consists of
\begin{equation}\label{formsimp}
\sum_{\ell=2}^i \binom{\,i-2\,}{\,\ell-2\,}m_{\ell}
\end{equation}
repeated edges of the form $(\tau,\sigma_{i,j}^{-1}\cdot \tau)$ for each permutation $\tau\in\mathcal{S}_n$. Here $\sigma_{i,j}\in\mathcal{S}_n$ is the cycle $n-i+1\to n-i+2\to n-i+3\to \cdots\to n-i+j\to n-i+1$. In these terms, the $\mathcal{S}_n$-action on edges is such that a permutation $\mu\in\mathcal{S}_n$ takes each (oriented) edge of the form $(\tau,\sigma_{i,j}\cdot \tau)$ homeomorphically onto an (oriented) edge of the form $(\tau\cdot\mu,\sigma_{i,j}\cdot \tau\cdot \mu)$.
\end{teo}

\begin{ejem}\label{ejemplo4t}{\em
$G(4,T)$ is obtained by splicing the graphs depicted in Figures~\ref{46}--\ref{22}, where a solid (dashed/double, respectively) arrow represents $m_2+2m_3+m_4$ ($m_2+m_3$/$m_2$, respectively) edges. As in Figures~\ref{3} and~\ref{d3}, the notation $(i,j,k,\ell)$ is simplified to $ijk\ell$ in Figures~\ref{46}--\ref{22} but, unlike the former figures (where edges are distributed by levels $E_i=\sqcup_j E_{i,j}$), edges are distributed by cycle lengths in the latter figures. (For $n\geq4$, graphs organized by levels are not planar nor as aesthetic as those in Figures~\ref{46}--\ref{22}.) Table~\ref{n4} summarizes the combinatorial structure, with levels $E_i$ corresponding to arrow styles, and where the first column gives the following information (which will the basis for the proof in the general case): From each $\mathcal{S}_4$-orbit of critical edges in $\mathcal{D}^4T$, select the unique representative $\left((a_0,a_1),a_2,a_3,a_4\right)$ satisfying $a_2<a_3<a_4$. The first column of Table~\ref{n4} then lists the conditions this representative must have in order for its $\mathcal{S}_4$-orbit to have the indicated level and cycle structure. Note that if $B$ denotes the stack of vertices $a_i$ blocked by the root vertex, then the level is read off from the cardinality of~$B$, whereas the cycle structure is read off from the cycle structure of the permutation determined by the vertices outside $B$. 
\begin{table}[ht!]
\centering
\begin{tabular}{|c|c|c|c|} \hline
{\footnotesize condition in $\mathcal{D}^4T$} & {\footnotesize $\mathcal{S}_4$-orbit} & {\footnotesize amount of repeated edges} & {\footnotesize level} \\ \hline
{\footnotesize $a_2=0$, $a_3=1$, $a_0<a_4<a_1$} & {\footnotesize 12 cycles of lenght two} \hspace{2mm}& {\footnotesize $m_2$} & {\footnotesize $E_2$} \\ \hline
{\footnotesize $a_2=0,$ $a_0<a_3<a_1<a_4$} & {\footnotesize12 cycles of length two} \hspace{2mm}& {\footnotesize $m_2+m_3$} & {\footnotesize $E_3$} \\
{\footnotesize $a_2=0,$ $a_0<a_3<a_4<a_1$} & {\footnotesize 8 cycles of length three} & {\footnotesize $m_2+m_3$} & {\footnotesize $E_3$}\\ \hline
{\footnotesize $a_0<a_2<a_1<a_3<a_4$} & {\footnotesize 12 cycles of length two} \hspace{2mm}& {\footnotesize $m_2+2m_3+m_4$} & {\footnotesize $E_4$}\\
{\footnotesize $a_0<a_2<a_3<a_1<a_4$} & {\footnotesize 8 cycles of length three} & {\footnotesize $m_2+2m_3+m_4$} & {\footnotesize $E_4$} \\
{\footnotesize $a_0<a_2<a_3<a_4<a_1$} & {\footnotesize 6 cycles of lenghth four} & {\footnotesize $m_2+2m_3+m_4$} & {\footnotesize $E_4$} \\ \hline
\end{tabular}  \caption{Edge types in $G(4,T)$.} \label{n4}
\end{table}
}\end{ejem}
\begin{center}
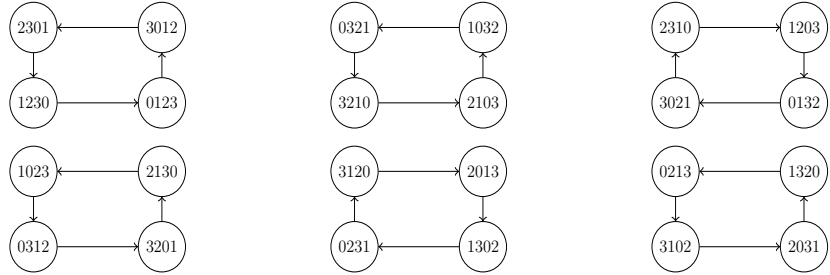
\begin{figure}[ht!]
\centering
\resizebox{11cm}{3.6cm}{
\begin{tikzpicture}[every node/.style={circle, draw, scale=.6}, scale=1.0, rotate = 180, xscale = -1]
\node (1) at ( -4,-.9) {1230};
\node (2) at (-2,-.9)  {0123};
\node (3) at ( -4,-2){2301};
\node (4) at (-2,-2) {3012};

\node (5) at ( -4,.1) {1023};
\node (6) at (-2, .1) {2130};
\node (7) at (-4, 1.2) {0312};
\node (8) at (-2, 1.2) {3201};

\node (9) at ( 1, -.9) {3210}; 
\node (10) at ( 3,-.9) {2103};
\node (11) at (3,-2) {1032};
\node (12) at (1,-2) {0321};

\node (13) at (1,.1)  {3120}; 
\node (14) at (3,.1)  {2013};
\node (15) at (1,1.2)  {0231};
\node (16) at (3,1.2)  {1302};

\node (17) at (6,.1)  {0213};
\node (18) at (8,.1)  {1320};
\node (19) at (6,1.2)  {3102};
\node (20) at (8,1.2)  {2031};

\node (21) at (6,-2) {2310};
\node (22) at (8,-2) {1203};
\node (23) at (6,-.9)  {3021};
\node (24) at (8,-.9)  {0132};

\draw [->] (21) -- (22); 
\draw [->] (23) -- (21);
\draw [->] (22) -- (24);
\draw [->] (24) -- (23);
\draw [->] (18) -- (17);
\draw [->] (17) -- (19);
\draw [->] (20) -- (18);
\draw [->] (19) -- (20);
\draw [->] (13) -- (14);
\draw [->] (15) -- (13);
\draw [->] (14) -- (16);
\draw [->] (16) -- (15);
\draw [->] (1) -- (2); 
\draw [->] (2) -- (4);
\draw [->] (4) -- (3);
\draw [->] (3) -- (1);
\draw [->] (5) -- (7);
\draw [->] (6) -- (5);
\draw [->] (8) -- (6);
\draw [->] (7) -- (8);
\draw [->] (9) -- (10);
\draw [->] (12) -- (9);
\draw [->] (11) -- (12);
\draw [->] (10) -- (11);
\end{tikzpicture}}
\caption{$E_{4,4}$ in $G(4,T)$; here and in the next two figures, arrows represent multiple edges as indicated in Example~\ref{ejemplo4t}.}
\label{46}
\end{figure}
\end{center}
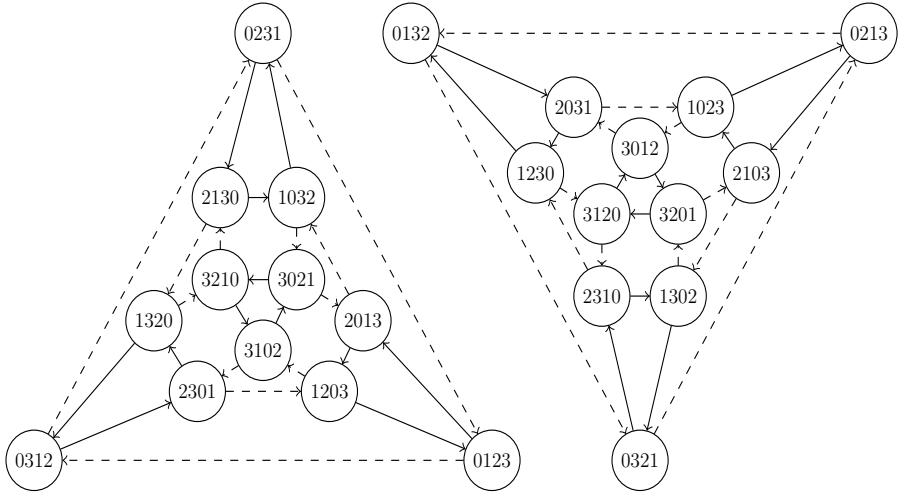
\begin{figure}[ht!]
	\centering
	\resizebox{12cm}{6.5cm}{
	\begin{tikzpicture}[every node/.style={circle, draw, scale=.6}, scale=1.0, rotate = 180, xscale = -1]
		\node (1) at ( 3.66, 3) {0231};
		\node (2) at ( 0.66, 8.2) {0312};
		\node (3) at ( 6.66, 8.2) {0123};
		\node (4) at ( 3.1, 5) {2130};
		\node (5) at ( 4.1, 5) {1032};
		\node (6) at ( 3.1, 6) {3210};
		\node (7) at ( 4.1, 6) {3021};
		\node (8) at ( 4.97, 6.5) {2013};
		\node (9) at ( 4.53, 7.36) {1203};
		\node (10) at ( 2.23, 6.5) {1320};
		\node (11) at ( 2.8, 7.36) {2301};
		\node (12) at ( 3.66, 6.86) {3102};
		\draw[->] (1) -- (4);
		\draw[->] (5) -- (1);
		\draw[->] (4) -- (5);
		\draw[dashed,->] (6) -- (4);
		\draw[dashed,->] (5) -- (7);
		\draw[->] (7) -- (6);
		\draw[->] (6) -- (12);
		\draw[->] (12) -- (7);
		\draw[dashed,->] (4) -- (10);
		\draw[dashed,->] (10) -- (6);
		\draw[dashed,->] (12) -- (11);
		\draw[->] (11) -- (10);
		\draw[->] (10) -- (2);
		\draw [->](2) -- (11);
		\draw[dashed,->] (7) -- (8);
		\draw[dashed,->] (8) -- (5);
		\draw[dashed,->] (9) -- (12);
		\draw[->] (8) -- (9);
		\draw[->] (9) -- (3);
		\draw[->] (3) -- (8);
		\draw[dashed,->] (11) -- (9);
		\draw[dashed,->] (1) -- (3);
		\draw[dashed,->] (3) to (2);
		\draw[dashed,->] (2) to (1);
\node (13) at ( 5.6, 3) {0132};
\node (22) at ( 11.6, 3) {0213};
\node (23) at ( 8.6, 8.2) {0321};
\node (14) at ( 7.73, 3.9) {2031};
\node (25) at ( 9.46, 3.9) {1023};
\node (18) at ( 8.1, 6.2) {2310};
\node (21) at ( 7.23, 4.7) {1230};
\node (20) at ( 10.06, 4.7) {2103};
\node (15) at ( 8.1, 5.2) {3120};
\node (16) at ( 9.1, 5.2) {3201};
\node (19) at ( 9.1, 6.2) {1302};
\node (17) at ( 8.6, 4.4) {3012};
\draw[dashed,->] (22) -- (13);
\draw[dashed,->] (13) -- (23);
\draw[dashed,->] (23) -- (22);
\draw[->] (13) -- (14);
\draw[->] (25) -- (22);
\draw[dashed,->] (14) -- (25);
\draw[->] (21) -- (13);
\draw[->] (14) -- (21);
\draw[->] (22) -- (20);
\draw[->] (20) -- (25);
\draw[->] (23) -- (18);
\draw[dashed,->] (15) -- (18);
\draw[dashed,->] (19) -- (16);
\draw[->] (18) -- (19);
\draw[->] (19) -- (23);
\draw[->] (16) -- (15);
\draw[dashed,->] (21) -- (15);
\draw[dashed,->] (16) -- (20);
\draw[->] (17) -- (16);
\draw[dashed,->] (17) -- (14);
\draw[->] (15) -- (17);
\draw[dashed,->] (25) -- (17);
\draw[dashed,->] (20) -- (19);
\draw[dashed,->] (18) -- (21);
	\end{tikzpicture}}
	\caption{$E_{4,3}$ (solid arrows) and $E_{3,3}$ (dashed arrows) in $G(4,T)$.}
	\label{38}
\end{figure}
\begin{figure}[ht!]
\centering
\resizebox{10.5cm}{9.5cm}{
\begin{tikzpicture}[every node/.style={circle, draw, scale=.6}, scale=1.0, rotate = 180, xscale = -1]
\node (1) at ( 0.5, 0.73) {1320};
\node (2) at ( 1.5, 2.73) {0321};
\node (3) at ( 1, 3.6) {0231};
\node (4) at ( -1.23, 3.6) {1230};
\node (5) at ( 2.5, 2.73) {0312};
\node (6) at ( 3, 3.6) {0213};
\node (7) at ( 2.5, 4.46) {0123};
\node (8) at ( 1.5, 4.46) {0132};
\node (9) at ( 4.5, 2.73) {1302};
\node (10) at ( 4, 3.6) {1203};
\node (11) at ( 2.5, 6.22) {1032};
\node (12) at ( 3, 5.36) {1023};
\node (13) at ( 4.5, 4.46) {2103};
\node (14) at ( 4, 5.36) {2013};
\node (15) at ( 5.5, 2.73) {2301};
\node (16) at ( 6, 3.6) {3201};
\node (17) at ( 5.5, 4.46) {3102};
\node (18) at ( 3, 7.08) {2031};
\node (19) at ( 4, 7.08) {3021};
\node (20) at ( 4.5, 6.22) {3012};
\node (21) at ( 6.5, .73) {2310};
\node (22) at ( 8.23, 3.6) {3210};
\node (23) at ( 2, 9.08) {2130};
\node (24) at ( 5, 9.08) {3120};
\draw[->] (2)  to[bend left=19] (1);
\draw[double,->] (2)  to[bend left=20] (3);
\draw[->] (4)  to[bend left=19] (3);
\draw[double,->] (1)  to[bend left=20] (4);
\draw[dashed,->] (5)  to[bend left=19] (2);
\draw[double,->] (6)  to[bend left=20] (5);
\draw[dashed,->] (7)  to[bend left=19] (6);
\draw[double,->] (7)  to[bend left=20] (8);
\draw[dashed,->] (3)  to[bend left=19] (8);
\draw[->] (9)  to[bend left=19] (5);
\draw[double,->] (9)  to[bend left=20] (10);
\draw[->] (6)  to[bend left=19] (10);
\draw[->] (12)  to[bend left=19] (7);
\draw[double,->] (11)  to[bend left=20] (12);
\draw[->] (8)  to[bend left=19] (11);
\draw[dashed,->] (14)  to[bend left=19] (12);
\draw[->] (13)  to[bend left=19] (14);
\draw[dashed,->] (10)  to[bend left=19] (13);
\draw[dashed,->] (18)  to[bend left=19] (11);
\draw[double,->] (19)  to[bend left=20] (18);
\draw[dashed,->] (20)  to[bend left=19] (19);
\draw[dashed,->] (15)  to[bend left=19] (9);
\draw[double,->] (16)  to[bend left=20] (15);
\draw[dashed,->] (17)  to[bend left=19] (16);
\draw[double,->] (13)  to[bend left=20] (17);
\draw[double,->] (20)  to[bend left=20] (14);
\draw[->] (17)  to[bend left=19] (20);
\draw[->] (15)  to[bend left=19] (21);
\draw[->] (16)  to[bend left=19] (22);
\draw[double,->] (21)  to[bend left=20] (22);
\draw[->] (24)  to[bend left=19] (19);
\draw[->] (23)  to[bend left=19] (18);
\draw[double,->] (24)  to[bend left=20] (23);
\draw[dashed,->] (21)  to[bend left=19] (1);
\draw[dashed,->] (23)  to[bend left=19] (4);
\draw[dashed,->] (24)  to[bend left=19] (22);
\draw[->] (1)  to[bend left=19] (2);
\draw[double,->] (3)  to[bend left=20] (2);
\draw[->] (3)  to[bend left=19] (4);
\draw[double,->] (4)  to[bend left=20] (1);
\draw[dashed,->] (2)  to[bend left=19] (5);
\draw[double,->] (5)  to[bend left=20] (6);
\draw[dashed,->] (6)  to[bend left=19] (7);
\draw[double,->] (8)  to[bend left=20] (7);
\draw[dashed,->] (8)  to[bend left=19] (3);
\draw[->] (5)  to[bend left=19] (9);
\draw[double,->] (10)  to[bend left=20] (9);
\draw[->] (10)  to[bend left=19] (6);
\draw[->] (7)  to[bend left=19] (12);
\draw[double,->] (12)  to[bend left=20] (11);
\draw[->] (11)  to[bend left=19] (8);
\draw[dashed,->] (12)  to[bend left=19] (14);
\draw[->] (14)  to[bend left=19] (13);
\draw[dashed,->] (13)  to[bend left=19] (10);
\draw[dashed,->] (11)  to[bend left=19] (18);
\draw[double,->] (18)  to[bend left=20] (19);
\draw[dashed,->] (19)  to[bend left=19] (20);
\draw[dashed,->] (9)  to[bend left=19] (15);
\draw[double,->] (15)  to[bend left=20] (16);
\draw[dashed,->] (16)  to[bend left=19] (17);
\draw[double,->] (17)  to[bend left=20] (13);
\draw[double,->] (14)  to[bend left=20] (20);
\draw[->] (20)  to[bend left=19] (17);
\draw[->] (21)  to[bend left=19] (15);
\draw[->] (22)  to[bend left=19] (16);
\draw[double,->] (22)  to[bend left=20] (21);
\draw[->] (19)  to[bend left=19] (24);
\draw[->] (18)  to[bend left=19] (23);
\draw[double,->] (23)  to[bend left=20] (24);
\draw[dashed,->] (1)  to[bend left=19] (21);
\draw[dashed,->] (4)  to[bend left=19] (23);
\draw[dashed,->] (22)  to[bend left=19] (24);
\end{tikzpicture}}
\caption{$E_{4,2}$ (solid arrows), $E_{3,2}$ (dashed arrows) and $E_{2,2}$ (double arrows) in $G(4,T)$.}
\label{22}
\end{figure}
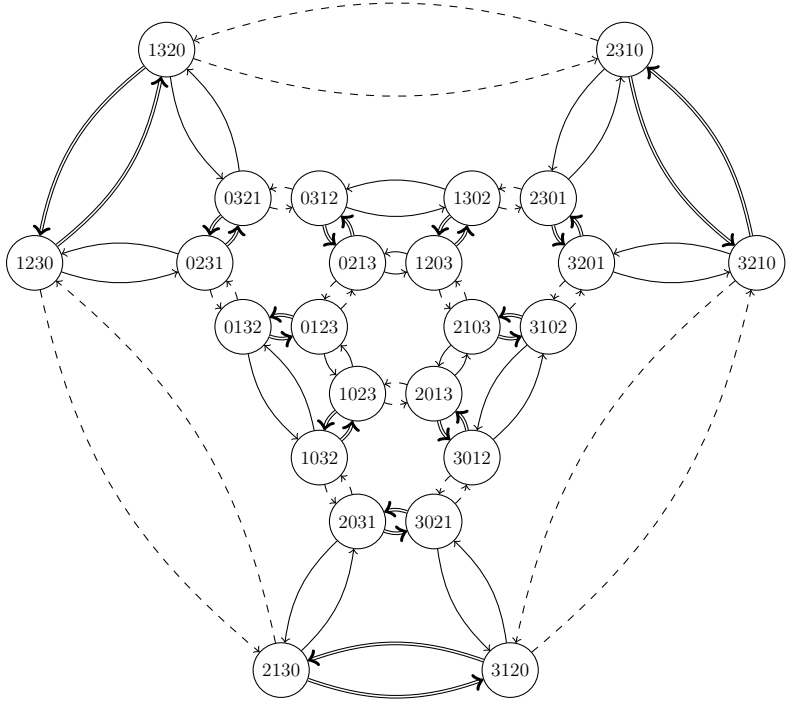

\begin{remark}{\em
As discussed in the proof of Theorem~\ref{n3} (and argued below for the general case), the canonical orientation (indicated in the figures above by the arrow direction) of each edge $(\tau,\sigma_{i,j}^{-1}\cdot\tau)\in E_{i,j}$, i.e., from $\tau$ to $\sigma_{i,j}^{-1}\cdot \tau$, is dictated by the monodromy of the covering space $\mathcal{D}^nT\to U\mathcal{D}^nT$ when each unordered critical 1-cell $\{(a_0,a_1),a_2,\ldots,a_n\}$ of $U\mathcal{D}^nT$ is (also canonically) oriented from $\{a_0,a_2,\ldots,a_n\}$ to $\{a_1,a_2,\ldots,a_n\}$.
}\end{remark}

Theorem~\ref{elcasotegeneral} implies that the $\mathcal{S}_n$-action on $G(n,T)$ restricts to an action on each ```sublevel''~$E_{i,j}$. In fact, as it will be clear from the proof below, each $E_{i,j}$ is a union of $\,\sum_{\ell=2}^i\binom{\,i-2\,}{\,\ell-2\,}m_\ell$ many full $\mathcal{S}_n$-orbits. Further, if we think of any $\mathcal{S}_n$-orbit in $E_{i,j}$ as a subgraph of $G(n,T)$, then its components are $j$-cycles in one-to-one correspondence with the quotient set of $\mathcal{S}_n$ by the subgroup generated by $\sigma_{i,j}$. In particular, each $\mathcal{S}_n$-orbit in $E_{i,j}$ is a set of $n!$ edges forming $\frac{n!}{j}$ pairwise disjoint $j$-cycles. On the other hand, while $G(n,T)$ has multiple edges, repetitions happen only within a single sublevel, and precisely as prescribed in Theorem~\ref{elcasotegeneral}. Indeed, as the reader can easily check, a pair of edges of $G(n,T)$ of respective types $(\tau_1,\sigma_{i,j}\cdot\tau_1)$ and $(\tau_2,\sigma_{r,s}\cdot\tau_2)$ can share end points  (i.e.~$\{\tau_1,\sigma_{r,s}\cdot\tau_1\}=\{\tau_2,\sigma_{r,s}\cdot\tau_2\}$) only if $\tau_1=\tau_2$, $i=r$ and $j=s$. Thus:

\begin{coro}\label{elcasotegeneralversionviejaahora coro}
\begin{enumerate}
\item[(1)] Multiple edges in $G(n,T)$ belong to a common sublevel $E_{i,j}$.
\item[(2)] Each $E_{i,j}$ contains $\sum_{\ell=2}^i \binom{\,i-2\,}{\,\ell-2\,}m_\ell$ edges with the same endpoints, for every distinct pair of endpoints of edges in $E_{i,j}$.
\item[(3)] If two different $\mathcal{S}_n$-orbits in $G(n,T)$ both form $\frac{n!}{\ell}$ disjoint $\ell$-cycles, then either they belong to different level sets $E_i$, or they belong to the same sublevel $E_{i,j}$ (in which case they are repeated $\mathcal{S}_n$-orbits).
\end{enumerate} 
\end{coro}

\begin{remark}\label{cantidadtotaldearistas}{\em
Theorem~\ref{elcasotegeneral} implies that the total number of edges in the graph $G(n,T)$ is 
$$n! \sum_{2\leq\ell\leq i\leq n} (i-1) \binom{\,i-2\,}{\,\ell-2\,}m_\ell.$$ 
A reader familiar with~\cite{fs} will note that the above formula (ignoring the factor $n!$, in view of Corollary~\ref{dim}.2) gives a simple expression for its unordered analogue, i.e., Theorem~4.1 in \cite{fs} in the case of trees. Here we remark that Farley-Sabalka's counting is done by subtracting a well controlled subset from the set of ways to distribute indistinguishable balls into distinguishable boxes when boxes are allowed to receive no balls. In contrast, our approach makes a direct counting based on what we believe is a new combinatorial formula: the number of ways to distribute indistinguishable balls into distinguishable boxes under special restrictions (one being that no box is allowed to receive zero balls). The combinatorial fact behind our formula is Vandermonde's identity. See Lemma~\ref{vandermonde}.
}\end{remark}

\begin{ejem}\label{siempresi}{\em
If $T$ is an ($n$-sufficiently subdivided) radial tree with essential vertex of degree~$d$ (so that $m_\ell=\binom{d-1}{\ell}$), then $G(n,T)$ is a connected graph with $n!$ vertices and $$n!\sum_{2\leq \ell\leq i\leq n}(i-1)\binom{i-2}{\ell-2}\binom{d-1}{\ell}$$ edges.
}\end{ejem}

\begin{proof}[Proof of Theorem~\ref{elcasotegeneral}]
In what follows $a=\{(a_0,a_1), a_2, \dots a_n\}$ stands for a critical unordered 1-cell of $U\mathcal{D}^nT$, and we let $c_a:=((a_0,a_1), a_2, \ldots, a_n)$ stand for the ordered representative of $a$ having
\begin{equation}\label{enorden}
a_2 < a_3 < \cdots < a_n.
\end{equation}

Let $E_{n-i}$ stand for the set of edges $e(c)\in E(G(n,T))$ coming from a critical 1-cell $c$ of $(\mathcal{D}^nT)'_1$ whose stack of vertices blocked by the root vertex $0$ has cardinality $i$. It is obvious that the $\mathcal{S}_n $-action~(\ref{raccion}) is closed on (the union of the cells in) each $E_{n-i}$, while the Classification Theorem~\ref{teoremadeclasificacion} implies that $E_{n-i}$ is nonempty only for $0\leq i\leq n-2$. Further, in terms of $c_a$ representatives, the condition $e(c_a)\in E_{n-i}$ (and therefore $e(c_a)\cdot\sigma\in E_{n-i}$, for all $\sigma\in\mathcal{S}_n$) is equivalent to the condition that
\begin{equation}\label{c1d2}
\mbox{the equality $a_\ell=\ell-2$ holds for $2\leq \ell\leq i+1$ but fails for $\ell=i+2$.}
\end{equation}

Let $E_{n-i,j}$ be the set of edges in $E_{n-i}$ of the form $e(c_a)\cdot\sigma$ with $\sigma\in\mathcal{S}_n$ and  
\begin{equation}\label{c2d2}
a_0 < \underbrace{a_{i+2} < a_{i+3} < \cdots < a_{i+j}} < a_1 < \underbrace{a_{i+j+1} < a_{i+j+2} < \cdots < a_n}.
\end{equation}
It is again obvious that each $E_{n-i,j}$ is a union of $\mathcal{S}_n$-orbits (which we count below), and that the Classification Theorem~\ref{teoremadeclasificacion} implies that $E_{n-i,j}$ is nonempty only for $2\leq j \leq n-i$ (i.e., the first underbraced sequence of inequalities above cannot be empty, but the second one can).

Using the notation introduced in the proof of Theorem~\ref{n3}, it is clear that, for $e(c_a)\in E_{i,j}$, the initial and final endpoints of $c_a$ (namely $(a_0,a_2,a_3,\ldots,a_n)$ and $(a_1,a_2,a_3,\ldots,a_n)$, respectively), support unique gradient paths
\begin{align}
(a_0,a_2,a_3,\ldots,a_n)\stackrel{\text{\tiny grad}}{\longrightarrow}&(i,0,1, \dots , i-1, i+1, \dots, n-1)\label{inicio}\\
(a_1,a_2,a_3,\ldots,a_n)\stackrel{\text{\tiny grad}}{\longrightarrow}&(i+j-1,0,1,\dots i-1,i,i+1,\ldots,i+j-2,i+j,i+j+1,\ldots,n-1)\nonumber\\
&=(i+j-1,0,1,\ldots, i+j-2,i+j,i+j+1,\ldots,n-1).\label{final}
\end{align}
In terms of~(\ref{0iden}),~(\ref{inicio}) corresponds to the cycle
$$
\tau_0=\left[\hspace{.3mm}1\to i+1\to i\to i-1\to \cdots\to 3\to 2\to 1\,\rule{0mm}{4mm}\right],
$$
while~(\ref{final}) corresponds to the cycle
$$
\tau_1=\left[\hspace{.3mm}1\to i+j\to i+j-1\to \cdots\to 3\to 2\to 1\,\rule{0mm}{4mm}\right].
$$
An elementary calculation gives $\tau_1=\sigma_{n-i,j}^{-1}\tau_0$, so that $e(c_a)$ has the asserted form. The corresponding fact for the $\mathcal{S}_n$-orbit generated by $e(c_a)$ follows from $\mathcal{S}_n$-equivariance.

\medskip
It remains to count the repeated edges in a given $E_{n-i,j}$, i.e., of $\mathcal{S}_n$-orbits making up $E_{n-i,j}$. We do this by counting the critical 1-cells $c_a$ with $e(c_a)\in E_{n-i,j}$, i.e., satisfying~(\ref{c1d2}) and~(\ref{c2d2}). With this in mind, we next spell out the structure (coming from the Classification Theorem~\ref{teoremadeclasificacion}) of critical 1-cells. In what follows $c_a$ stands for a critical 1-cells in $E_{n-i,j}$ (thus satisfying~(\ref{enorden}),~(\ref{c1d2}) and~(\ref{c2d2})). 

First of all,~(\ref{c1d2}) says that, no matter what $c_a$ is, the $i$ vertices $a_2,a_3,\ldots,a_{i+1}$ must be piled up\footnote{Here and below, piling up is done according to~(\ref{enorden}), without leaving empty spots in the tree (as in~(\ref{c1d2})).} in a stack\footnote{Stacks like these are uniquely determined, as the tree has been assumed to be $n$-sufficiently subdivided.} blocked by the root vertex 0. Diversity comes only from the edge $(a_0,a_1)$ ---$a_0$ must be an essential vertex--- and the vertices
\begin{equation}\label{lacola}
a_{i+2},a_{i+3},\ldots,a_n
\end{equation}
satisfying~(\ref{c2d2}). The first $j-1$ vertices in~(\ref{lacola}) are smaller than $a_1$, and are piled up in, say, $p$ stacks ($p\geq1$) each of which is blocked by $a_0$. For the remaining vertices, i.e., the last $n-i-j$ vertices in~(\ref{lacola}), we have:
\begin{enumerate}[(i)]
\item all of them are  larger than $a_1$;
\item the first $y_0$ of them form a stack blocked by $a_1$ (here $y_0\geq0$);
\item the remaining ones (if any) are piled up in, say, $q$ stacks ($q\geq0$) each blocked by $a_0$.
\end{enumerate}
For these conditions to actually determine the cell $c_a$, we must specify:
\begin{enumerate}[(i)]\addtocounter{enumi}{3}
\item the labels of the edges adjacent to $a_0$ that support the piles, and
\item the distribution ---subject to the conditions above--- of vertices in piles.
\end{enumerate}

In the description above, the total number of $a_0$-local $T$-branches\footnote{An $a_0$-local $T$-branch is a sequence of edges $(a_0,b), (b,c), (c,d),\ldots$ (this implies $a_0<b<c<d<\cdots$).} carrying entries of $c_a$ is $p+q+1$ (the final ``${}+1$'' accounts for the branch determined by the edge $(a_0,a_1)$, independently of whether $y_0=0$). So, the number of different \emph{types} of critical 1-cells $c_a$ ---i.e., families of such 1-cells sharing a common information in (i)--(iv) above--- that involve the essential vertex $a_0$ is
\begin{equation}\label{parrafo}
\binom{d(a_0)-1}{\,p+q+1}.
\end{equation}
In particular, the number of different types of such critical 1-cells $c_a$ is $m_{p+q+1}$. To finish, we would need to count the number of critical 1-cells $c_a$ of a given type (i.e.~accounting for the information in~(v)). However, the combined counting is presented best (and yields the simple formula~(\ref{formsimp})) by taking a slightly different viewpoint of the description above of $c_a$, namely, by stressing on the role of the sum $p+q+1$, rather than on the role of $p$ and $q$ independently (note that  $p+q+1\in\{2,\ldots,n-i\}$). Such an approach is motivated by Lemma~\ref{vandermonde} below, a combinatorial fact that, to the best of our knowledge, has not appeared in print before. We thus make a brief detour in the course of the proof of Theorem~\ref{elcasotegeneral}.

\begin{lema}\label{vandermonde}
Fix integers $f$, $r$ and $s$ satisfying $r\geq s\geq2$ and $r>f>0$. Let $N(f,r,s)$ denote the number of ways to distribute $r$ indistinguishable balls into $s$ distinguishable boxes (say boxes are numbered from 1 to $s$) so that
\begin{itemize}
\item no box is empty, and 
\item the number of balls in some initial segment of boxes is $f$.
\end{itemize}
Then $N(f,r,s)=\binom{r-2}{s-2}$. (Note that $N(f,r,s)$ is thus independent of $f$.)
\end{lema}

By the expression ``initial segment of boxes'' in the statement of Lemma~\ref{vandermonde}, we mean the collection of balls in boxes 1 through $k$, for some $k\in\{1,\ldots,s-1\}$ ($k$ cannot be $s$, as $r>f$).

\begin{ejem}{\em
For $r=6$, $s=3$ and $f=4$, the distribution that has $3$ balls in the first box, $2$ balls in the second box, and $1$ ball in the last box is ``illegal'' (the second item in Lemma~\ref{vandermonde} is not fulfilled) and, thus, not accounted for $N(4,6,3)$. Instead, the distribution that has 4 balls in the first box, 1 ball in the second box, and 1 ball in the third box is allowable and is to be accounted for in $N(4,6,3)$. Likewise, the distribution that has 2 balls in each box is allowable.

}\end{ejem}

\begin{remark}\label{compa-racion}{\em
Lemma~\ref{vandermonde} should be compared with the following standard fact: \emph{The binomial coefficient $\binom{r-1}{s-1}$ counts the number of ways to distribute $r$ indistinguishable balls into $s$ distinguishable boxes in such a way that (boxes are allowed to have more than one element, but) no box can be empty.}
}\end{remark}

\begin{proof}[Proof of Lemma~\ref{vandermonde}]
An allowable distribution has a total of $f$ balls in the first $k$ boxes, for some $1\leq k\leq s-1$. Taking into account the formula in Remark~\ref{compa-racion}, this means that 
$$
N(f,r,s)=\sum_{k=1}^{s-1}\binom{f-1}{k-1}\binom{r-f-1}{s-k-1}=\binom{r-2}{s-2},
$$
where the last equality is Vandermonde's identity.
\end{proof}

\noindent{\it Proof conclusion of Theorem~\ref{elcasotegeneral}.} Fix $\ell\in\{2,\ldots,n-i\}$. Consider critical 1-cells $c_a$ whose entries occupy a total of $\ell$ $a_0$-local $T$-branches (i.e., $\ell=p+q+1$ in the discussion above). As discussed in the paragraph containing~(\ref{parrafo}), there are $m_\ell$ types of such cells, and we only need to prove that each such type encompasses
\begin{equation}\label{tagging}
\binom{n-i-2}{\ell-2}
\end{equation}
cells. Recall that each of such cells is determined by the distribution (piling up) of its vertices on the (already) selected $a_0$-local $T$-branches. In the distribution process, vertices (i.e., balls in Lemma~\ref{vandermonde}) behave as being indistinguishable, for the distribution has to be done according to~$(\ref{enorden})$. On the other hand, the $\ell$ (i.e., $s$ in Lemma~\ref{vandermonde}) $a_0$-local $T$-branches (i.e., boxes in Lemma~\ref{vandermonde}) receiving vertices are distinguishable, since this determines the actual critical 1-cell. There are $n-i-1$ (i.e., $r$ in Lemma~\ref{vandermonde}) vertices to distribute, namely those in~(\ref{lacola}). However, a characteristic that does not seem to fit in the situation of Lemma~\ref{vandermonde} is the fact that, in the distribution, one and only one of the branches ---the one determined by the edge $(a_0,a_1)$--- is allowed to be empty, for the parameter $y_0$ in the first part of the proof is allowed to be zero. Nevertheless  $a_1$ is also a vertex to be distributed (together with its corresponding edge $(a_0,a_1)$), which fixes the possible emptiness of that box, except that we would need to flag the branch where this holds, for $a_1$ is not ``canonically'' ordered in~(\ref{c2d2}). But the flagging is forced on us (by the parameter $j$): as discussed in the first part of the proof, the first $j-1$ vertices in~(\ref{lacola}) have to fill up an initial segment of stacks, after which, the next branch has to be the flagged one, i.e., the one carrying the edge $(a_0,a_1)$ (and possibly some other additional vertices). The situation now fits perfectly Lemma~\ref{vandermonde}: there are $n-i$ indistinguisable balls (i.e., $a_1$ has now been added) to distribute in $\ell$ distinguishable boxes (none of which can end up being empty, by our setup), with $j-1$ (i.e., $f$ in Lemma~\ref{vandermonde}) of the balls filling up an initial segment of boxes. The desired number~(\ref{tagging}) follows then directly from Lemma~\ref{vandermonde}.
\end{proof}

\section{The higher topological complexity of $\mathcal{C}^nT$}
Throughout this section $n$ stands for an integer greater than 1, and $T$ stands for a tree with at least one essential vertex. Theorem~\ref{1}, Corollary~\ref{dim} and Remark~\ref{hshiausy} show that the configuration space $\mathcal{C}^nT$ has the homotopy type of a cell complex of dimension $\ell=\ell(n,T)=\min\{\left[\frac{n}{2}\right], m\}$, where $m=m(T)$ stands for the number of essential vertices of $T$. In particular, the homotopy dimension of $\mathcal{C}^nT$ satisfies
\begin{equation}\label{porarriba}
\hdim(\mathcal{C}^nT)\leq\ell,
\end{equation}
so that Proposition~\ref{ulbTCnrararar} yields
\begin{equation}\label{iiddhgsusi}
\cat(\mathcal{C}^nT)\leq\ell\quad\mbox{and}\quad \TC_s({\mathcal{C}^nT})\leq s\ell\quad\mbox{for $s\geq2$.}
\end{equation}
This section's goal is to show that all inequalities in~(\ref{porarriba}) and~(\ref{iiddhgsusi}) are in fact equalities, with the single exception of the second inequality in~(\ref{iiddhgsusi}) when $T$ is the $Y$-graph and $n=2$ ---for which $\mathcal{C}^2Y\simeq S^1$ and $\TC_s(\mathcal{C}^2Y)=s-1$. 

\begin{remark}\label{casoexcepcional}{\em
The exceptional case above is reminiscent of the more general formula
$$
\TC_s\left(\bigvee_NS^1\right)=\begin{cases}
s-1, & \text{if $N=1$;}\\ s, & \text{if $N>1$,}
\end{cases}
$$
which holds for $s\geq2$ (see~\cite[Corollary~1.4]{MR3117387}). Indeed, Theorem~\ref{1}, Theorem~\ref{deformacionfuerteespecial}.1, Remark~\ref{hshiausy} and Example~\ref{siempresi} imply that, if $T$ is a radial tree (i.e.~$m=1$), then $\mathcal{C}^nT$ has the homotopy type of a wedge of
\begin{equation}\label{exceptional1}
1+n!\left(-1+\sum_{2\leq j\leq i\leq n}(i-1)\binom{i-2}{j-2}\binom{d-1}{j}\right)
\end{equation}
circles, where $d$ is the degree of the only essential vertex of $\hspace{.2mm}T$. It is then elementary to see that~(\ref{exceptional1}) is a positive integer that equals 1 precisely for $n=2$ and $d=3$. Likewise, by Theorem~\ref{farbergeneralized}, $\mathcal{C}^2T$ has the homotopy type of a wedge of 
\begin{equation}\label{exceptional2}
2\sum_{d(v)\geq3}\binom{d(v)-1}{2}-1
\end{equation}
circles, with~(\ref{exceptional2}) being 1 precisely when $T$ is the $Y$-graph. Note further that Theorem~\ref{1}, Corollary~\ref{dim} and Remarks~\ref{hshiausy} and~\ref{cantidadtotaldearistas} imply that $\mathcal{C}^3T$ is homotopy equivalent to a wedge of 
\begin{equation}\label{noexcepcional}
1+6\left(-1+\sum_{\substack{v\text{ essential}\\2\leq j\leq i\leq 3}}(i-1)\binom{i-2}{j-2}\binom{d(v)-1}{j}\right)
\end{equation}
circles. Since~(\ref{noexcepcional}) is at least 7, we get equality in all instances of~(\ref{porarriba}) and~(\ref{iiddhgsusi}) with $n=3$.
}\end{remark}

We start with a standard algebraic-topology consequence of Theorem~\ref{1} and Remark~\ref{inclusion}.
\begin{coro}\label{vertices}
Let $v_{1}, \ldots, v_{m}$ be the essential vertices of $T$. For each $i= 1, \ldots, m$ let $Y_i$ be a $Y$-graph embedded in $T$ around a small neighborhood of $v_{i}$ (so that $Y_i\cap Y_j=\varnothing$ for $i\neq j$). Then there are cohomology classes $\alpha_{i}\in H^{1}(\mathcal{C}^2T)$ such that $\alpha_{i}|_{\mathcal{C}^2Y_j}=\delta_{ij}\in H^{1}(\mathcal{C}^2Y_j)=\mathbb{Z}$, where $\delta_{ij}$ is the Kronecker delta.
\end{coro}

For $i=1,\ldots, \ell$, consider the strictly commutative diagram
\begin{equation}\label{classes}
\xymatrix{\prod\limits_{j=1}^{\ell} \mathcal{C}^2Y_j\ar[r]^(.55){\Psi}\ar[d]_(.6){\mathrm{pr}_{i}} & \mathcal{C}^{2\ell}T\ar[d]^{\Phi_i}\\ \mathcal{C}^2Y_i\ar@{^(->}[r]& \mathcal{C}^2T,} 
\end{equation} 
where $\Psi((x_1,x_2), \ldots, (x_{2\ell-1},x_{2\ell}))=(x_1,x_2, \ldots, x_{2\ell-1},x_{2\ell})$, $\Phi_{i}(x_1, \ldots, x_{2\ell})=(x_{2i-1},x_{2i})$, and $\mathrm{pr}_{i}$ is the projection onto the $i$-th factor. Consider the cohomology classes $\alpha_{ij}\in H^{1}(\mathcal{C}^{2\ell}T)$ defined by $$\alpha_{ij}=\Phi_{i}^{*}(\alpha_{j})\quad\mbox{for}\quad i,j\in\{1,2, \ldots, \ell\}.$$ 
Since $\mathcal{C}^2T$ has the homotopy type of a (banana) graph, we see
\begin{equation}\label{productostriviales00}
\alpha_{ij}\alpha_{ik}=0\text{ \ for any indices $i,j,k$}.
\end{equation}
Further, the commutativity of~(\ref{classes}) ensures
\begin{equation}\label{ensures}
\Psi^{*}(\alpha_{ij})=0, \mbox{\ \  if \ \ }i\neq j,
\end{equation}
whereas $\Psi^{*}(\alpha_{ii})\neq 0$ for, in fact, the product 
\begin{equation}\label{topcohomologicalclass}
t^{*}:=\Psi^{*}(\alpha_{11}\alpha_{22}\cdots \alpha_{\ell\ell})=\Psi^{*}(\alpha_{11}) \Psi^{*}(\alpha_{22})\cdots\Psi^{*}(\alpha_{\ell\ell})
\end{equation}
generates the top cohomology group of the torus $\prod_{j=1}^{\ell}\mathcal{C}^2Y_j$. In particular, $\alpha_{11}\alpha_{22}\cdots\alpha_{\ell\ell}\neq 0$, so that the cup-length $\cl(\mathcal{C}^{2\ell}T)$ is no smaller than $\ell$. The latter condition is easily extended to $\mathcal{C}^nT$:

\begin{prop}\label{cl}
The cup-length of $\mathcal{C}^nT$ satisfies $\,\cl(\mathcal{C}^nT)\geq\ell$.
\end{prop}
\begin{proof}
The map $p:\mathcal{C}^nT\longrightarrow \mathcal{C}^{2\ell}T$ given by $p(x_1,\ldots,x_n)=(x_1,\ldots,x_{2\ell})$ admits a homotopy section (if $e$ is an edge in $T$ connecting a vertex $v$ of degree one to some other vertex, then we can shrink $e$ away from $v$ to make room for $n-2\ell$ additional particles), so the induced map $p^{*}$ in cohomology is injective. The result follows since $0\neq p^*(\alpha_{11} \cdots \alpha_{\ell\ell})=p^*(\alpha_{11}) \cdots p^*(\alpha_{\ell\ell})$.
\end{proof}

\begin{coro}\label{cat}
The inequality in~(\ref{porarriba}) and the first inequality in~(\ref{iiddhgsusi}) are in fact equalities.
\end{coro}
\begin{proof}
$\ell\leq\cl(\mathcal{C}^nT)\leq\cat(\mathcal{C}^nT)\leq \hdim(\mathcal{C}^nT)\leq\ell$.
\end{proof}

In view of the discussion in Remark~\ref{casoexcepcional}, the only assertion that remains to be proved is the equality $\TC_s(\mathcal{C}^nT)=s\ell$ for $s\geq2$, $n\geq4$ and $m\geq2$ (so $\ell\geq2$), hypotheses that are in force from this point on.

We need the following preparations. Let $t$ be the top homology class in the domain of the map $\Psi$ in~(\ref{classes}), and write $z=\Psi_{*}(t)\in H_{\ell}(\mathcal{C}^{2\ell}T)$. From~(\ref{topcohomologicalclass}) we get
\begin{equation}\label{productequal1}
1=\langle t^*,t\rangle=\langle\Psi^{*}(\alpha_{11}\alpha_{22}\cdots\alpha_{\ell\ell}),t\rangle = \langle\alpha_{11}\alpha_{22}\cdots\alpha_{\ell\ell},z\rangle,
\end{equation}
whereas~(\ref{ensures}) yields  
\begin{equation}\label{productequalzero}
0=\langle 0,t\rangle=\langle\Psi^*(\alpha_{1r_{1}}\alpha_{2r_{2}}\cdots\alpha_{\ell r_\ell}),t\rangle=\langle\alpha_{1r_{1}}\alpha_{2r_{2}}\cdots\alpha_{\ell r_\ell},z\rangle,
\end{equation}
if $r_{i}\neq i$ for some $i=1,\ldots, \ell$. For a permutation $\sigma\in \mathcal{S}_\ell$, consider the commutative diagram
\begin{equation}\label{diagramsigma}
\xymatrix{\prod\limits_{j=1}^{\ell} \mathcal{C}^2Y_j\ar[r]^(.56){\Psi}\ar[dr]_{\Psi_{\sigma}\,=\,g_\sigma\hspace{.35mm}\circ \Psi} & \mathcal{C}^{2\ell}T\ar[d]_{g_\sigma}\ar[dr]^{\Phi_{\sigma(i)}}&\\
&\mathcal{C}^{2\ell}T\ar[r]_{\Phi_{i}}&\mathcal{C}^2T,} 
\end{equation} 
where $g_\sigma(x_{1},x_{2}, \ldots, x_{2\ell-1},x_{2\ell})=(x_{2\sigma(1)-1}, x_{2\sigma(1)}, \ldots, x_{2\sigma(\ell)-1}, x_{2\sigma(\ell)})$. Lastly, set $z_{\sigma}:=(g_\sigma)_{*}(z)$.
\begin{prop}\label{productsigma}
For $\sigma\in\mathcal{S}_\ell$,
\begin{equation}
\nonumber
\langle\alpha_{1r_{1}}\alpha_{2r_{2}}\cdots\alpha_{\ell r_{\ell}},z_{\sigma}\rangle=\begin{cases}
\pm1, & \text{if } r_{j}=\sigma(j) \text{ for each } j;\\
0, & \text{otherwise.}\end{cases}
\end{equation}
\end{prop}
\begin{proof}
The commutativity of the right-hand triangle in~(\ref{diagramsigma}) gives $g_\sigma^*(\alpha_{ij})=\alpha_{\sigma(i)j}$, so
$$\langle\alpha_{1r_{1}}\alpha_{2r_{2}}\cdots\alpha_{\ell r_{\ell}},z_{\sigma}\rangle=
\langle\alpha_{\sigma(1)r_{1}}\alpha_{\sigma(2)r_{2}}\cdots\alpha_{\sigma(\ell) r_{\ell}},z\rangle=
\begin{cases}
\pm1, & \text{if } r_{j}=\sigma(j) \text{ for each } j;\\
0, & \text{otherwise.}
\end{cases}$$
in view of~(\ref{productequal1}) and (\ref{productequalzero}).
\end{proof}

It will be convenient to first analyze $\TC_s(\mathcal{C}^nT)$ for $s=2$.
\begin{lema}\label{TCF2l}
We have $\TC_2(\mathcal{C}^{2\ell}T)=2\ell.$
\end{lema}
\begin{proof}
For a cohomology class $\alpha\in H^*(\mathcal{C}^{2\ell}T)$, let $\overbar{\alpha}\in H^{*}(\mathcal{C}^{2\ell}T)\otimes H^{*}(\mathcal{C}^{2\ell}T)$ stand for the zero-divisor $\overbar{\alpha}= \alpha\otimes 1 - 1\otimes \alpha$. By~(\ref{iiddhgsusi}) and Proposition~\ref{ulbTCnrararar}, it suffices to check 
\begin{equation}\label{product}
\left(\overbar{\alpha_{11}}\cdots\overbar{\alpha_{(\ell-1)(\ell-1)}}\overbar{\alpha_{\ell\ell}}\right) \left(\overbar{\alpha_{12}}\cdots\overbar{\alpha_{(\ell-1)\ell}}\overbar{\alpha_{\ell1}}\right)\neq 0.
\end{equation}
In what follows, the cohomology class $\overbar{\alpha_{\ell1}}$ will also be denoted as $\overbar{\alpha_{\ell(\ell+1)}}$. Put $[\ell]:=\{1,2,\ldots,\ell\}$, then
\begin{align*}
(\overbar{\alpha_{11}}\cdots\overbar{\alpha_{\ell\ell}})( \overbar{\alpha_{12}}\cdots\overbar{\alpha_{\ell(\ell+1)}})
=&\left( \sum_{\substack{S_{1}\subseteq [\ell]}} \pm \prod_{i\in S_{1}}\alpha_{ii}\otimes \prod_{i\notin S_{1}}\alpha_{ii}\right) \left(\sum_{\substack{S_{2}\subseteq [\ell]}}\pm \prod_{i\in S_{2}}\alpha_{i(i+1)}\otimes \prod_{i\notin S_{2}}\alpha_{i(i+1)}\right)\\
=&\sum_{\substack{S_{1}\subseteq[\ell]\\S_{2}\subseteq[\ell]}}\pm\left(\prod_{i\in S_1}\alpha_{ii}\right)\left(\prod_{i\in S_2}\alpha_{i(i+1)}\right)\otimes\left( \prod_{i\notin S_1}\alpha_{ii}\right)\left(\prod_{i\notin S_2}\alpha_{i(i+1)}\right)\\
=&\sum_{\substack{S\subseteq[\ell]}}\pm\left(\,\prod_{i\in S}\alpha_{ii}\right)\left(\,\prod_{i\notin S}\alpha_{i(i+1)}\right)\otimes\left(\,\prod_{i\notin S}\alpha_{ii}\right)\left(\,\prod_{i\in S}\alpha_{i(i+1)}\right),
\end{align*}
where the last equality holds because an $(S_1,S_2)$-indexed summand in the next-to-last expression vanishes unless $S_1\cap S_2=\varnothing=S_1^c\cap S_2^c$, in view of~(\ref{productostriviales00}). In particular, if $\tau\in\mathcal{S}_\ell$ is the cycle given by $\tau(i)=i+1$ for $i<\ell$, and $\tau(\ell)=1$, we get
\begin{align*}
\left\langle\rule{0mm}{4mm}(\overbar{\alpha_{11}}\cdots\overbar{\alpha_{\ell\ell}})\right.(\overbar{\alpha_{12}}&\cdots\overbar{\alpha_{\ell(\ell+1)}}), \left.\rule{0mm}{4mm}z\otimes z_{\tau}\right\rangle\\
=&\sum_{\substack{S\subseteq[\ell]}}\pm \left\langle\left(\,\prod_{i\in S}\alpha_{ii}\right)\left(\,\prod_{i\notin S}\alpha_{i(i+1)}\right)\otimes \left(\,\prod_{i\notin S}\alpha_{ii}\right)\left(\,\prod_{i\in S}\alpha_{i(i+1)}\right), z\otimes z_{\tau}\right\rangle\\
=&\sum_{\substack{S\subseteq[\ell]}}\pm \left\langle\left(\,\prod_{i\in S}\alpha_{ii}\right)\left(\,\prod_{i\notin S}\alpha_{i(i+1)}\right),z\right\rangle \left\langle\left(\,\prod_{i\notin S}\alpha_{ii}\right)\left(\,\prod_{i\in S}\alpha_{i(i+1)}\right), z_{\tau}\right\rangle\\
=&\pm\left\langle\,\prod_{i=1}^{\ell}\alpha_{ii},z\right\rangle\left\langle\,\prod_{i=1}^{\ell}\alpha_{i(i+1)},z_{\tau}\right\rangle=\,\pm 1,
\end{align*}
where the last two equalities use Proposition~\ref{productsigma}. This yields~(\ref{product}).
\end{proof}

\begin{coro}\label{TCFn} 
We have $\TC_2({\mathcal{C}^nT})=2\ell$.
\end{coro}
\begin{proof}
We have noted that the projection $p:\mathcal{C}^nT\longrightarrow \mathcal{C}^{2\ell}T$ given by $p(x_1,\ldots,x_n)=(x_1,\ldots,x_{2\ell})$ admits a homotopy section. Thus any cartesian power $p^s\colon (\mathcal{C}^nT)^s\to(\mathcal{C}^{2\ell}T)^s$ induces a monomorphism in cohomology groups. The result then follows from Proposition~\ref{ulbTCnrararar} and~(\ref{iiddhgsusi}) by noticing that the pullbacks under $p\times p$ of the classes $\overbar{\alpha_{ii}}$ and $\overbar{\alpha_{i(i+1)}}$ for $1\leq i\leq \ell$, are zero-divisors with non-zero product.
\end{proof}

\begin{teo}\label{elteofinal}
For $s\geq2$ (and in the presence of the assumptions set forth after the proof of Corollary~\ref{cat}), $\TC_s(\mathcal{C}^nT)=s\ell$.
\end{teo}

Theorem~\ref{elteofinal} follows from Proposition~\ref{yacasi} below in the same way as Corollary~\ref{TCFn} follows from Lemma~\ref{TCF2l}.
\begin{prop}\label{yacasi}
For $s\geq2$, $\TC_s(\mathcal{C}^{2\ell}T)=s\ell$.
\end{prop}
\begin{proof}
We can assume $s\geq3$. For  $i,j\in[\ell]$, set 
$$
\widehat{\alpha_{ij}}= \overbar{\alpha_{ij}}\otimes \underbrace{1\otimes \cdots\otimes 1}_{s-2} \in H^{*}(\mathcal{C}^{2\ell}T)^{\otimes s}.
$$
By~(\ref{iiddhgsusi}) and Proposition~\ref{ulbTCnrararar}, it suffices to check  the non-triviality of the product
\begin{align}\label{higherproduct}
\left(\rule{0mm}{4mm}\widehat{\alpha_{11}}\cdots\widehat{\alpha_{\ell\ell}}\right) \left(\rule{0mm}{4mm} \widehat{\alpha_{12}}\cdots\widehat{\alpha_{\ell(\ell+1)}}\right)&
\left(\,\prod_{i=1}^{\ell}(1\otimes\alpha_{ii}\otimes 1\otimes \cdots\otimes 1 -1\otimes 1\otimes \alpha_{ii}\otimes 1\otimes\cdots\otimes 1)\right)\cdots\nonumber\\
&\cdots\left(\,\prod_{i=1}^{\ell}(1\otimes1\otimes\cdots\otimes\alpha_{ii}\otimes 1 -1\otimes1\otimes\cdots\otimes 1\otimes\alpha_{ii})\right).
\end{align}
From the proof of Lemma~\ref{TCF2l}, the product of the first $2\ell$ factors in~(\ref{higherproduct}) takes the form  
\begin{align*}
\sum_{\substack{S\subseteq[\ell]}}\pm\left(\,\prod_{i\in S}\alpha_{ii}\right)\left(\,\prod_{i\notin S}\alpha_{i(i+1)}\right)\otimes\left(\prod_{i\notin S}\alpha_{ii}\right)\left(\prod_{i\in S}\alpha_{i(i+1)}\right)\otimes 1\otimes\cdots \otimes 1,
\end{align*}
and since $H^{\ell+1}(\mathcal{C}^{2\ell}T)=0$ (in view of Corollary~\ref{cat}),~(\ref{higherproduct}) becomes
$$
\sum_{\substack{S\subseteq[\ell]}}\pm\prod_{i\in S}\alpha_{ii}\prod_{i\notin S}\alpha_{i(i+1)}\otimes\prod_{i\notin S}\alpha_{ii}\prod_{i\in S}\alpha_{i(i+1)}\otimes \prod_{i=1}^{\ell}\alpha_{ii}\otimes\cdots \otimes \prod_{i=1}^{\ell}\alpha_{ii}.
$$
The non-triviality of the latter expression follows from Proposition~\ref{productsigma} by evaluating at the homology class $z\otimes z_{\tau}\otimes z\otimes\cdots \otimes z$.
\end{proof}


\end{document}